\documentclass{imanum}
\usepackage{setspace}
\usepackage{amsmath,amssymb}
\usepackage[utf8]{inputenc}
\usepackage{mathtools}
\usepackage{enumerate}
\usepackage{color}
\usepackage{amsthm}
\usepackage{mathrsfs,bbm}
\usepackage{epstopdf}
\usepackage{graphicx,graphics}
\definecolor{refblue}{rgb}{0,0,0.75} 
\definecolor{refblueb}{rgb}{0,0,1} 
\definecolor{refgreen}{rgb}{0,0.5,0} 
\definecolor{grey}{rgb}{0.7,0.7,0.7} 
\usepackage{hyperref}
\hypersetup{ 
	colorlinks   = true, %Colours links instead of ugly boxes 
	urlcolor     = refblue, %Colour for external hyperlinks 
	linkcolor    = refblueb, %Colour of internal links 
	citecolor   = refblue%Colour of citations 
}
\jno{ }
\received{ }
\revised{ }

\DeclareMathAlphabet\mathbfcal{OMS}{cmsy}{b}{n}

\newcommand{\R}{\mathbb{R}}
\newcommand{\C}{\mathbb{C}}

\newcommand{\boldsymb}{\boldsymbol }
\newcommand{\bupsilon}{\boldsymb \upsilon}

\newcommand{\mXG}{{ \mathrm{\boldsymbol{X}}_\Gamma}}
\newcommand{\mVG}{{\mathrm{\boldsymbol{V}}_\Gamma}}

\DeclareMathAlphabet\mathbfcal{OMS}{cmsy}{b}{n}

\renewcommand{\Re}{\mathrm{Re}\,}

\newcommand{\pt}{\partial_t}
\newcommand{\pttau}{\partial_t^\tau}
\DeclareOldFontCommand{\bf}{\normalfont\bfseries}{\mathbf}
\DeclareOldFontCommand{\it}{\normalfont\itshape}{\mathit}
\newcommand{\e}{\mathrm{e}}
\newcommand{\epsmu}{_{\varepsilon,\mu} }

\newcommand{\p}{\boldsymb p}

\newcommand{\bginc}{\boldsymbol{g}^{\textnormal{inc}}}
\newcommand{\Einc}{\boldsymbol{E}^\textnormal{inc}}
\newcommand{\Hinc}{\boldsymbol{H}^\textnormal{inc}}
\newcommand{\bJ}{\boldsymbol{J}}

\newcommand{\bpsi}{{\boldsymbol{\psi}}}
\newcommand{\bphi}{\boldsymbol{\phi}}
\newcommand{\bwphi}{\widehat{\boldsymbol{\phi}}}
\newcommand{\bvar}{\boldsymbol{\varphi}}
\newcommand{\bwphip}{\widehat{\boldsymbol{\phi}}^+}
\newcommand{\bwphim}{\widehat{\boldsymbol{\phi}}^-}
\newcommand{\bphip}{{\boldsymbol{\phi}}^+}
\newcommand{\bphim}{{\boldsymbol{\phi}}^-}

\newcommand{\bAinvs}{\bA^{-1}(s)}
\newcommand{\bvarphi}{{\boldsymbol{\varphi}}}

\newcommand{\bxi}{\boldsymbol{\xi}}

\newcommand{\Om}{\Omega}
\newcommand{\bwginc}{\widehat{\boldsymbol{g}}^{\mathrm{inc}}}
\newcommand{\bg}{\boldsymbol{g}}

\newcommand{\bwg}{\widehat{\boldsymbol{g}}}

\newcommand{\bwvarphi}{\widehat{\boldsymbol{\varphi}}}
\newcommand{\bwpsi}{\widehat{\boldsymb\psi}}

\newcommand{\cqb}{\boldsymb b}
\newcommand{\cqc}{\boldsymb c}
\newcommand{\cqg}{\boldsymb g}
\newcommand{\cqC}{\boldsymb C}
\newcommand{\cqK}{\boldsymb K}
\newcommand{\cqH}{\boldsymb H}
\newcommand{\cqI}{\boldsymb I}
\newcommand{\cqL}{\boldsymb L}

\newcommand{\cqW}{\boldsymb W}
\newcommand{\cqX}{\boldsymb X}
\newcommand{\cqY}{\boldsymb Y}

\newcommand{\wchi}{\widehat{\chi}}

\newcommand{\bA}{\boldsymb{A}}

\newcommand{\divG}{\operatorname{div}_\Gamma}
\DeclareMathOperator{\curl}{\mathbf{curl}}

\newcommand{\norm}[1]{\left\lVert#1\right\rVert}

\newcommand{\abs}[1]{\left|#1\right|}
\newcommand{\jmp}[1]{[#1]}
\newcommand{\avg}[1]{\{#1\}}

\newcommand{\dleb}[1]{\,\mathrm{d}{#1}}

\renewcommand{\Re}{\operatorname{Re}}
\newcommand{\dist}{\operatorname{dist}}
\newcommand{\inc}{_\textnormal{inc}}

\newcommand{\tot}{^\textnormal{tot}}
\newcommand{\normal}{\boldsymbol{\nu}}

\newcommand{\gaT}{{\gamma}_{T}}

\newcommand{\A}{\boldsymb  A}

\newcommand{\B}{\boldsymb H}

\newcommand{\bC}{\boldsymb C}
\newcommand{\E}{\boldsymb E}
\newcommand{\wE}{\widehat{\boldsymb E}}
\newcommand{\wH}{\widehat{\boldsymb H}}
\newcommand{\I}{\mathrm{I}}

\newcommand{\bId}{\textbf{I}\!\hspace{0.7pt}\textbf{d}}
\renewcommand{\H}{\boldsymb H}
\newcommand{\K}{\boldsymb K}
\newcommand{\bL}{\boldsymb L}
\newcommand{\bR}{\boldsymb R}
\newcommand{\U}{\mathbfcal U}

\newcommand{\V}{\boldsymb V}
\newcommand{\W}{\mathbfcal W}

\newcommand{\X}{\boldsymb X}

\newcommand{\Cald}{\boldsymb C}

\renewcommand{\a}{\boldsymb a}
\renewcommand{\b}{\boldsymb b}

\newcommand{\g}{g}

\renewcommand{\u}{\boldsymb u}
\renewcommand{\v}{\boldsymb \upsilon}

\newcommand{\x}{ x}
\newcommand{\y}{ y}

\renewcommand{\x}{\boldsymb x}
\renewcommand{\y}{\boldsymb y}
\newcommand{\bx}{\boldsymb x}

\newcommand{\Ga}{\Gamma}
\newcommand{\eps}{\varepsilon}
\newcommand{\bone}{\mathbbm{1}}
\newcommand{\Ctrace}{C_\Gamma}
\newcommand{\ctrace}{c_\Gamma}

\usepackage[toc,header]{appendix}
\shorttitle{Electromagnetic scattering from dispersive materials}
\title{Time-dependent electromagnetic scattering from dispersive materials}
\author{
{\sc
J\"org Nick \thanks{ Email: joerg.nick@math.ethz.ch}
} \\[2pt]
Seminar für Angewandte Mathematik, ETH Zürich, CH-8092 Z\"urich, Switzerland \\[6pt]
{\sc  Selina Burkhard}\thanks{Email: selina.burkhard@kit.edu}\\[2pt]
Institut für Angewandte und Numerische Mathematik, Karlsruher Institut für Technologie (KIT), D-76131 Karlsruhe, Germany\\[6pt]
{\sc Christian Lubich}\thanks{Email: lubich@na.uni-tuebingen.de}\\[2pt]
Mathematisches Institut, Universität Tübingen, Auf der Morgenstelle, D-72076 Tübingen,\\
Germany  
}

\begin{document}

\maketitle
 
\begin{abstract}
{This paper studies time-dependent electromagnetic scattering from metamaterials that are described by dispersive material laws. We consider the numerical treatment of a scattering problem in which a dispersive material law, for a causal and passive homogeneous material, determines the wave-material interaction in the scatterer. The resulting problem is nonlocal in time inside the scatterer and is posed on an unbounded domain. Well-posedness of the scattering problem is shown using a formulation that is fully given on the surface of the scatterer via a time-dependent boundary integral equation. Discretizing this equation by convolution quadrature in time and boundary elements in space yields a provably stable and convergent method that is fully parallel in time and space. Under regularity assumptions on the exact solution we derive error bounds with explicit convergence rates in time and space. Numerical experiments illustrate the theoretical results and show the effectiveness of the method.}
% Keywords
{electromagnetic scattering, dispersive material laws, time-dependent partial differential equations, convolution quadrature, boundary element method} 
\end{abstract}

\section{Introduction and setting}
%Mathematical models for dispersive electromagnetic material laws, in particular negative index materials, have been subject of much interest in recent years. 
Since the pioneering work of \cite{V68}, dispersive materials and their interaction with electromagnetic waves have attracted much scientific interest. The use of metamaterials promises to advance many applications in the context of optical devices and imaging. A collection of applications is found in \cite[Section 5]{L16}.

A survey of the mathematical literature is given by \cite{L16} and a coherent presentation of basic mathematical results is given by \cite{CJK17}. 
Following that paper, we require the causality principle and homogeneity for the metamaterials, and we consider the fundamental class of (strongly) passive materials.
%, cf. \cite{CJK17}. 

In the approach to time-dependent scattering taken here, the scattering problem posed in the exterior domain and the dispersive bounded scattering object
is reduced to a time-dependent boundary integral equation for the tangential traces of the electric and magnetic fields. From the tangential traces, the electromagnetic fields inside and outside the scatterer are then obtained via representation formulas. We show well-posedness of the boundary integral equation and the scattering problem for causal and passive homogeneous dispersive materials.
We use convolution quadrature and boundary elements for the numerical discretization and provide a fully discrete error analysis. Numerical experiments illustrate the theoretical results and show the effectiveness of the method.

%A natural extension to the setting investigated in \cite{NKL2020,NKL22,N22} is to allow the scattering objects to be nonlocal in time. 

	The numerical simulation of wave propagation problems on
	exterior domains by discretizing time-dependent boundary integral equations with convolution quadrature in time and boundary elements in space originates from \cite{L94}. This approach has been taken up in the numerical literature both in the acoustic
case, e.g. \cite{LS09,BLS15,BS09,BanjaiRieder,BLN20,S16}, and in the electromagnetic case, e.g. \cite{ChenMonkWangWeile,BBSV13,ChanMonk2015,KL17,NKL22}. In \cite{DES21}, a convolution quadrature discretization has been applied to dispersive electromagnetic material laws in combination with finite volume techniques.

Further related literature is, e.g.,  \cite{QiuS16}, which considers the acoustic wave equation and a reformulation as retarded boundary integral equation. The discretization is provided by a Galerkin 
semi-discretization in space and convolution quadrature in time. In contrast to us, this paper uses evolution equation techniques for the fully discrete system. In \cite{EFHS21}, a formulation of an acoustic wave transmission problem with mixed boundary conditions as a retarded potential integral equation is derived and wellposedness is shown.
\cite{ChanMonk2015} solve dielectric scattering with a homogeneous penetrable obstacle, by using boundary integral methods and convolution quadrature.

For a recent overview on convolution quadrature with applications to scattering problems, we refer to \cite{BanS22}.

\subsection{Dispersive Maxwell's equations on a single domain $\Omega$}
Let $\Omega\subset \mathbb R^3$ be an interior or exterior domain. We are interested in time-dependent (possibly dispersive) electromagnetic waves, which are modeled by \emph{Maxwell's equations} 
(here assumed with vanishing current and charge)
\begin{align}\label{eq:MW-complete}
\begin{split}
	\partial_t \boldsymb D -\curl \H &= 0  
	\\ \partial_t \boldsymb B +\curl \E&=  0 
\end{split}
\qquad \text{in } \ \Omega.
\end{align}
%(The current $\boldsymb J$ and the charge $\rho$ are assumed to vanish and therefore do not appear in this formulation.)  
These equations are complemented by the material laws
\begin{align}\label{eq:material law}
 \boldsymb D = \varepsilon_0\E + \boldsymb P ( \E) , \quad 
 \boldsymb B  =\mu_0\boldsymb H + \boldsymb M ( \H),
\end{align}
with the constant permittivity  $ \varepsilon_0$ and permeability  $\mu_0$ of vacuum and
with the polarization field $\boldsymb P$ and the magnetization field $\boldsymb M$. 
%Formulated in the time domain and for spatially invariant material laws, 
For homogeneous materials, as will be considered here,
these fields are of the form of a temporal convolution
\begin{align}\label{eq:pol}
\boldsymb P (\E)(t) &= \varepsilon_0\int_{0}^t\chi_e(t')\E(t-t') \mathrm d t',
\\
\label{eq:magnet}\boldsymb M (\H)(t) &= \mu_0\int_{0}^t\chi_m(t')\H(t-t') \mathrm d t',
\end{align}
with the scalar susceptibility kernels $\chi_e$ and $\chi_m$. 
%\subsection{Examples of time-varying retarded material laws}
\subsection{Examples of retarded material laws}\label{sec:rml_examples}
We present various material laws, which can be found in \cite{CJK17,BKN11}. The reaction of material, different from vacuum,  is non-instantaneous when exposed to stress. It depends on the past, it has a memory.

%\subsection*{Linear combinations of exponential kernels} 
\begin{subequations}\label{eq:sus-kernels-all}
A common example used in the literature is the Debye model, for which the susceptibility kernel is given by
\begin{equation}\label{eq:exp-kernel}
	\chi_e(t) = \beta \e^{-\lambda t}, \quad \text{for} \quad \lambda, \beta > 0,
\end{equation}
with relaxation parameter ${1}/{\lambda}$. Here the distraction of the material depends stronger on the more recent past, this is captured by the exponential damping. Kernels in physical applications often consist of sums of the type given above or, more generally, are completely monotonic functions
(cf.~\cite{Wid41})
$$
\chi_e(t)= \int_0^\infty \e^{-\lambda t} \beta(\dleb\lambda) \qquad\text{with a positive measure $\beta$}.
$$

Another class of material laws  simultaneously damps the system and introduces a temporal delay, by some fixed $t_* > 0$. The analysis of such systems has received significant attention, for example in \cite{NP06} and, for scattering problems, in \cite{P12}. In our setting, such models are described by the convolution with the shifted Heaviside function
\begin{equation}\label{eq:heaviside-kernel}
%	\chi_e(t) = \delta_{0}(t-t_*). \quad 
\chi_e(t)= \alpha_1 +\alpha_2	\Theta(t-t_*) = \begin{cases}
\alpha_1 +\alpha_2, \quad & t \ge t_*\\
\alpha_1, \quad& t < t_*.
\end{cases}
\end{equation}
We always assume that the parameters satisfy $ \alpha_1 \ge \alpha_2>0$.
In \cite{NP06}, the respective material law in the acoustic setting has been shown to be exponentially stable for $ \alpha_1 >\alpha_2>0$. The authors further construct, for the converse case of $\alpha_2 > \alpha_1$, arbitrary small shifts $t_*$ that destabilize the system. In the next sections we will show that the corresponding electromagnetic scattering  problem is well-posed and stable under the stated constraints.

Popular models for the propagation of light and its interaction with matter are the Drude and Lorentz models in nanophotonics; see~\cite{BKN11,CJK17}.
The Drude model can be used to model metal under the influence of an external electric field. Then, the conduction electrons behave as charged free particles, they form an ideal classical gas. 
%All interactions with other electrons and core ions are absorbed in a phenomenological 
With  the collision frequency $\gamma_D$ and the plasma frequency $\omega_D$,  the temporal susceptibility kernel is given by
\begin{equation}\label{eq:drude-kernel}
	\chi(t) = \frac{\omega_D^2}{\gamma_D}\left(1-\e^{-\gamma_D t}\right).
\end{equation}
The Drude model has to be modified for metals in the visible regime, where interband transitions of the electrons occur due to higher photon energies. In this case Lorentz oscillators provide a simple model: with 
%	\\	is "limited when it comes to the description of metals in the visible regime. There, higher photon energies induce interband transitions of the electrons. Additional Lorentz oscillators provide a simple model for such processes as well as effects associated with bound charges."
%	Lorentz; let 
\begin{equation}\label{eq:lorentz-kernel}
	\chi(t) = \frac{\beta_L}{\lambda_L}\e^{-\frac{\alpha_L}{2} t }\sin(\lambda_L t), \qquad \lambda_L = \sqrt{\omega_L^2 - \frac{\alpha_L}{4}},
\end{equation}
where  $\omega_L$ is the resonance frequency, 
$0 < \alpha_L < 4\omega_L^2$ is the damping coefficient and $\beta_L > 0$ gives the strength.

More involved models contain fractional derivatives, such as the Havriliak-Negami model, which models dielectric relaxation in complex systems, cf.~\cite{WW10} and~\cite{G15}: with positive coefficients $\beta$ and $\gamma$ and the exponent $0<\eta < 2$,
\begin{alignat}{4} \label{eq:fractional}
	\eps_0 \big(	(1+\gamma)\pt+ \beta \pt^{1+\eta}\big) \E - \curl \big(1+\beta\pt^\eta\big)\B=0.
	%		\varepsilon^-(\partial_t)&\,\partial_t\E^{{-}} &&-\curl \H^{-} &&= 0,  
	%		\\ \mu^-(\partial_t)&\,\partial_t\H^{{-}} &&+\curl \E^{-}&&=  0 .
\end{alignat}
\end{subequations}
This can be reformulated as \eqref{eq:material law} with \eqref{eq:pol}, where $\chi_e(t)$ is the convolution kernel given by its Laplace transform $\widehat \chi_e (s) = \gamma/(1+\beta s^\eta)$.

%Defining fractional derivatives as convolutions with weakly singular kernels \cite{Mai10}, we have for $0 < \eta \le 1$ that
%\begin{equation}\label{eq:frac-kernel}
%	\pt^\eta f(t) =\begin{cases}
%	k * f(t), \quad & 0 < \eta < 1,\\
%	\pt, \quad & \eta =1,
%	\end{cases}
%\end{equation}
%with 
%\begin{equation*}
%k(t) = \frac{1}{\Gamma(1-\eta)} t^{-\eta}.
%\end{equation*}
%Then a model used in physical applications, cf.~\cite{WW10}, including fractional kernels for the polarization corresponds to the temporal equation 
%% $\frac{\beta}{1+s^\alpha}$
%\begin{alignat*}{4}
%	\big(	(1+\gamma)\pt+ \beta \pt^{1+\eta}\big) \E - \curl \big(1+\beta\pt^\eta\big)\B=0.
%	%		\varepsilon^-(\partial_t)&\,\partial_t\E^{{-}} &&-\curl \H^{-} &&= 0,  
%	%		\\ \mu^-(\partial_t)&\,\partial_t\H^{{-}} &&+\curl \E^{-}&&=  0 .
%\end{alignat*}
%This is also known as the Havriliak-Negami model and has various applications, cf., \cite{G15}. Examples mentioned therein are thermal relaxation glasses, machining applications, dielectric relaxation in complex systems and vegetables, or mechanical relaxation in viscoelastic rods. 
%%\begin{equation*}
%%	\mathcal{L}(t \mapsto t^{a})(s) = \frac{\Gamma(a+1)}{s^{a+1}}, \qquad \Re a > -1
%%\end{equation*}	
%\end{subequations}

\subsection{The time-dependent scattering problem}
%In order to arrive at a full scattering problem, we 
We decompose the complete space $\mathbb R^3$ into the exterior domain $\Omega^+$, the interior (bounded) domain $\Omega^-$ and the interface $\Gamma =\partial \Omega^+= \partial \Omega^-$, which yields the disjoint union $\mathbb R^3 = \Omega^- \cup \Gamma \cup \Omega^+$. Inside the scatterer, i.e. in the bounded domain $\Omega^-$, we enforce a retarded material law and couple it with Maxwell's equations with physical parameters corresponding to a vacuum in the exterior domain $\Omega^+$. We then arrive at the following equations in their respective domains:\\ \ \\
	\begin{minipage}{0.5\textwidth}
		\centering{\text{In the interior domain $\Omega^{-}$}}:
		\begin{alignat*}{4}
		\varepsilon_0&\partial_t\E^{{-}} &&+ \boldsymb P ( 	\partial_t\E^{-}) &&-\curl \H^{-} &&= 0,  		\\ 
			\mu_0&\partial_t \ \H^{-} &&+ \boldsymb M (\partial_t\H^{-})&&+\curl \E^{-}&&=  0 .
		\end{alignat*}
	\end{minipage}
\begin{minipage}{0.5\textwidth}
			\centering{\text{In the exterior domain $\Omega^{{+}}$}}:
			\begin{align}\label{eq:td-scattering-domain}
			\begin{split}
			\varepsilon_0	\partial_t\E^{+} -\curl \H^{+} &= 0,  
			\\ \mu_0 \partial_t \boldsymb H^{+} +\curl \E^{+}&=  0 .
			\end{split} 
			\end{align}
		\end{minipage}
\		\\
Initially, we assume the system to be at rest with vanishing electromagnetic fields in the interior and the exterior. The system is excited by an exterior incoming electromagnetic wave $(\E\inc,\H\inc)$, a solution of the exterior Maxwell's equations with support initially away from the surface $\Gamma$ of the scatterer. The unknown exterior fields $(\E^+,\H^+)$ are referred to as the \emph{scattered fields}.  They uniquely identify, together with the incoming wave, the total electromagnetic fields via $\E\tot = \E^++\Einc$ and $\H\tot = \H^++\Hinc$. Inside  the scatterer there is only the initially vanishing scattered wave $(\E^-,\H^-)$, making such a distinction unnecessary.  

Along the interface of the scatterer $\Gamma=\partial \Omega^{\pm}$, we enforce continuity of the tangential components of the total electric and magnetic fields, which reads
\begin{align}\label{eq:td-transmission-conditions}
\begin{aligned}
\gamma_T\E^{-} &= \gamma_T\E^{+} + \gamma_T \E^{+}_{\text{inc}}, \\
\gamma_T \H^{-} &= \gamma_T \H^{+} + \gamma_T \H^{+}_{\text{inc}}
\end{aligned}\quad\ \text{on} \ \Gamma.
\end{align}

The numerical treatment of this scattering problems needs to overcome the following main challenges of the problem above:
\begin{itemize}
\item The material laws \eqref{eq:pol}--\eqref{eq:magnet} are nonlocal in time and therefore require, for general susceptibilities $(\chi_e,\chi_m)$, the whole
history of the solution at any time $t$, which leads to a memory tail with standard time-stepping discretizations.
\item The exterior domain $\Omega^+$ is unbounded.
\end{itemize}

For the ease of presentation we consider a retarded material law inside the scatterer and vacuum in the unbounded domain. The case of retarded material laws in both domains is a straightforward extension of this work.

\subsection{Outline and contributions of the paper}

%We extend the known analysis of electromagnetic scattering problems to material laws inside of the scatterer with frequency dependent  permittivities and permeabilities. Here, the knowledge of the explicit dependence on the parameters $\eps^\pm, \mu^\pm$ in the scatterer $\Omega^-$ and on the unbounded domain $\Omega^+$ in the bounds is needed.
%
%In this situation, we are also interested in approximating the electromagnetic fields inside of the scatterer. Taking all four boundary densities into account leads to a bigger and more complicated system, for which we derive boundedness and coercivity results.  
%
%
%%% Unterschied zur Arbeit von Chan und Monk
%\subsection*{\emph{Distinction from \cite{ChanMonk2015} and \cite{KL17}}} 
The present problem formulation is partly inspired by \cite{ChanMonk2015}, which gives the first numerical analysis for time-dependent electromagnetic scattering from dielectric penetrable obstacles. In this paper, we go beyond the existing literature by a thorough numerical analysis for scattering from dispersive materials, which are described by  non-local convolutional material laws in the time domain and frequency-dependent permittivities and permeabilities in the Laplace domain. The mathematical theory describing such materials has been extensively developed in the last years, see for example \cite{CJK17} for an excellent overview. 

We transfer the techniques developed in \cite{KL17} and \cite{NKL22}, which in turn originate in the acoustic analogues of \cite{BLS15} and \cite{BLN20}), respectively. On the analytical side, we show that the assumptions made on the mathematical models for dispersive materials lead to well-posed boundary integral equations. On the numerical side, employing a temporal discretization based on the convolution quadrature method combined with a boundary element space discretization, we provide the first provably stable and convergent numerical method for time-dependent electromagnetic scattering from dispersive materials based on time-dependent boundary integral equations. In the following, we give an outline and discuss the contributions of each section.

In Section~\ref{sec:framework}, we recall the foundation of dispersive Maxwell's equations and describe the framework of passive material laws used in the subsequent sections. Lemma~\ref{lem:examples} shows that all examples from the introduction are included in the setting. 

 Section~\ref{sec:time-harmonic-transmission} formulates and analyses a basic dispersive time-harmonic transmission problem, for which a central bound for the electromagnetic fields is shown in Lemma~\ref{lem:transmission}. 
%This result implies new bounds for the potential operators in Lemma~\ref{lem:transmission}. 
As a consequence, bounds for the potential and boundary operators corresponding to the time-harmonic dispersive Maxwell's equations are deduced. Moreover, the fundamental Calder\'on operator is constructed for the dispersive Maxwell's equations, which is differently scaled than its dielectric counterpart in the previous work by \cite{KL17}. Assuming passivity of the material law, we obtain the crucial time-harmonic coercivity result of Lemma~\ref{lem:B-coercivity}. 

In Section~\ref{sec:th-scattering}, we apply the previously established operators to derive a well-posed and stable time-harmonic boundary integral equation and prove the equivalence to the time-harmonic scattering problem of interest in Proposition~\ref{prop:wp-time-harm-scattering}. Moreover, $s$-specific bounds are shown which estimate the solution of the scattering problem in terms of the incoming wave. Assuming a stronger passivity assumption on the material law, we obtain simplified bounds of all operators, which can be transported into the time-domain. 
%From this point on, we assume the strong passivity formulated in Assumption~\ref{strongly-passive}, which is satisfyed by all examples in the introduction. Under weaker assumptions, one could repeat the analysis and obtain similar results, which we omit for the readability of the paper. 

Section~\ref{sec:td-scattering} carries the time-harmonic analysis over to the time-domain. The time-dependent boundary integral equation is formulated in \eqref{eq:bie-A-t} and its central properties are collectively shown in Theorem~\ref{th:time-dependent-well-posedness}. In particular, we show the well-posedness of the boundary integral equation, the equivalence to the time-dependent scattering problem and give estimates on the solution in terms of the incoming waves. All of these results are the direct consequence of their time-harmonic counterparts.  

In Section~\ref{section:time semi-discrete} we apply a convolution quadrature time discretization, based on the Radau IIA Runge--Kutta methods, to the boundary integral equation, which yields a temporally discrete scheme for the approximation of the scattering problem. We introduce some basic results surrounding the convolution quadrature method and crucial notation for the error analysis in the subsequent section, but omit the formulation of semi-discrete error bounds.    

Section~\ref{sec:full} introduces the spatial discretization based on Raviart--Thomas boundary elements and cites the best-approximation result used in this paper. The main part of this section consists of the error analysis, which leads to the main result of Theorem~\ref{th:error-bounds}. Here, the bulk of the previous analysis enters, however the structure of the proof is carried over from \cite[Theorem 6.1]{NKL22}. A complication on the way towards pointwise estimates away from the boundary is overcome by requiring additional regularity on the data (by following the ideas of the proof of Lemma~\ref{lem:eps-mu-s} and \eqref{eq:bound-U-h-dispersive}). 

Finally in Section~\ref{sec:numerics} we describe the numerical experiments conducted for the present setting. A fractional material law is used with a simple domain to compute empirical convergence rates in space and time, which illustrate the error bounds of the previous sections. An example with two cubes  demonstrates the use of the method and closes the paper.  

\section{Reformulation of the problem and mathematical framework}
\label{sec:framework}

\subsection{Reformulation of the time-dependent scattering problem}
Via the Laplace transforms $(\widehat\chi^{\pm}_e(s),\widehat\chi^{\pm}_m(s))$ of the susceptibility kernels $(\chi^{\pm}_e(t),\chi^{\pm}_m(t))$,
%(with $\chi^+_e(t)=0$ and $\chi^+_m(t)=0$ in \eqref{eq:td-scattering-domain}), 
we define the functions
%, corresponding to frequency depending physical parameters of Maxwell's equations
\begin{equation}
\varepsilon^{\pm}(s)= \varepsilon_0 (1+\widehat \chi^{\pm}_e (s)),\quad\ 	\mu^{\pm}(s)= \mu_0 (1+\widehat \chi^{\pm}_m (s)),
\end{equation}
which are the Laplace transforms of the distributions $\varepsilon_0(\delta + \chi^{\pm}_e)$ and 
$\mu_0 (\delta + \chi^{\pm}_m)$
(with Dirac's delta distribution).
We use the Heaviside notation for temporal convolution: for a function $g$ defined on the real line, 
\begin{equation} \label{Heaviside-epsmu}
\varepsilon^{\pm}(\pt)\g = (\mathcal{L}^{-1}\varepsilon^{\pm}) * \g, \quad\ \mu^{\pm}(\pt)\g = (\mathcal{L}^{-1}\mu^{\pm}) * \g,
\end{equation}
where $\mathcal{L}^{-1}$ denotes taking the inverse Laplace transform.
(We will later use this notation also for temporal convolutions related to other Laplace transforms.)
We then arrive at the following reformulation of the scattering problem: \\ 

	\begin{minipage}{0.5\textwidth}
		\centering{\text{In the interior domain $\Omega^{-}$}}:
		\begin{alignat*}{4}
		\varepsilon^-(\partial_t)&\,\partial_t\E^{{-}} &&-\curl \H^{-} &&= 0,  
		\\ \mu^-(\partial_t)&\,\partial_t\H^{{-}} &&+\curl \E^{-}&&=  0 .
		\end{alignat*}
	\end{minipage}
	\begin{minipage}{0.5\textwidth}
		\centering{\text{In the exterior domain $\Omega^{{+}}$}}:
		\begin{align}\label{eq:td-scattering-domain-2}
		\begin{split}
	\varepsilon^+(\partial_t)\,	\partial_t\E^{+} -\curl \H^{+} &= 0,  
		\\ \mu^+(\partial_t)\, \partial_t \boldsymb H^{+} +\curl \E^{+}&=  0 .
		\end{split} 
		\end{align}
	\end{minipage}
\ \\
This is completed by enforcing continuity of the tangential parts of the electromagnetic fields along the boundary, as in \eqref{eq:td-transmission-conditions}.

\subsection{Passivity conditions for the dispersive permittivities $\varepsilon^{\pm}(s)$ and permeabilities $\mu^{\pm}(s)$}

As is explained in \cite[after Definition~2.5]{CJK17},
passivity and causality %and the high frequency behavior 
of the time-varying material law result from the following property : 
%For $\Re s >0$ it holds that
\begin{equation}\label{passive}
\Re \bigl(  \varepsilon^{\pm}(s)s \bigr) > 0 \ \ \text{ and } \ \ \Re \bigl( \mu^{\pm}(s)s\bigr) > 0 
\qquad\text{for}\ \   \Re s >0.
\end{equation}
This will be assumed throughout this paper.
%Details on the consequences of this assumption can be found in  \cite[after Definition~2.5]{CJK17}.
For our purposes it will sometimes be useful to assume a stronger passivity condition: 
%for every $\sigma >0$, there exists $c_\sigma>0$ such that
\begin{align}\label{strongly-passive}
\Re \bigl(  \varepsilon^{\pm}(s)s \bigr) \ge \varepsilon_0 \Re s \ \ \text{ and } \ \ \Re \bigl(  \mu^{\pm}(s)s\bigr) \ge  \mu_0 \Re s 
\qquad\text{for}\ \ \Re s >0.
\end{align}
This condition is equivalent to $\Re( \widehat\chi^\pm_e(s)s)\ge 0$ and $\Re( \widehat\chi^\pm_m(s)s)\ge 0$ for $\Re s>0$.\\
We further assume a bound for $\varepsilon^{\pm}(s)$ and $\mu^{\pm}(s)$: for every $\sigma >0$, there exists $M_\sigma <\infty$ such that
\begin{equation}\label{eq:bound-eps-mu}
\abs{\varepsilon^{\pm}(s)}\le M_\sigma \eps_0\quad \text{and} \quad \abs{\mu^{\pm}(s)}\le M_\sigma \mu_0
\qquad\text{for}\quad \Re s \ge \sigma >0.
\end{equation}

%The material law approaches the physical properties of vacuum for $s\rightarrow \infty$, namely $\varepsilon^{\pm}(s)\rightarrow  \varepsilon_0$ and $\mu^{\pm}(s)\rightarrow \mu_0$ for $s\rightarrow \infty$ (c.f.  \cite{CJK17}). Since these functions are furthermore analytic on the closure of $\mathbb C_+=\{s\in\mathbb C | \Re s >0 \}$, there exist constants $C_{\varepsilon}$ and $C_\mu$ independent of $s$, such that
%\begin{equation}\label{eq:bound-eps-mu}
%\abs{\varepsilon^{\pm}(s)}\le C_{\varepsilon}\quad \text{and} \quad \abs{\mu^{\pm}(s)}\le C_{\mu}.
%\end{equation}

\begin{lemma}\label{lem:examples}
	All examples of \eqref{eq:sus-kernels-all} satisfy the strong passivity condition \eqref{strongly-passive} and the bound \eqref{eq:bound-eps-mu}.
\end{lemma}
\begin{proof}
	\noindent\eqref{eq:exp-kernel} For the Debye model with $\lambda>0$ and $\beta>0$ we have, for $\Re s > 0$,
	\begin{equation*}
	\wchi_e(s)=(\mathcal{L}\chi_e)(s) = \frac{\beta}{s +\lambda} \quad \text{ and hence}\quad \Re\bigl(\widehat{\chi}_e(s)s\bigr) = \frac{\beta}{\abs{s+\lambda}^2}\big(\abs{s}^2 +\lambda\Re s\big) \ge 0.
	\end{equation*}
More generally, by the same argument we also obtain the strong passivity for completely monotonic susceptibility kernels $\chi_e$.
	
	%	\begin{equation*}
%	\frac{\beta}{\abs{s +\lambda}} \le \frac{\beta}{\lambda},
%	\end{equation*}
%	such that we can bound $\eps(s)$ independently of $s$.
	
	\smallskip
	\eqref{eq:heaviside-kernel} The Laplace transform of the susceptibility kernel corresponding to the shifted Heaviside function $\chi_e(t)= \alpha_1+\alpha_2\Theta(t-t_*)$
	reads
	\begin{equation*}
	\wchi_e(s) = s^{-1}\left(\alpha_1+\alpha_2\e^{-t_* s}\right) ,
	\end{equation*}
	for which we obtain the strong passivity under the condition $\alpha_1\ge\alpha_2$.
%	 and from
%	\begin{align*}
%	\Re(s\eps(s)) = \Re s + \alpha_1+\alpha_2\Re \e^{-t_* s} \ge \Re s.
%	\end{align*}
%	Estimating slightly differently yields
%	\begin{align*}
%		1/T > e^{-t_*/T}
%	\end{align*}
%	\begin{align*}
%    	-\log(T) > -t_*/T
%    \end{align*}
%	For $\log(\Re s) > -t_* \Re s$ $$-\frac{\log(\sigma)}{\sigma} < t_* $$ 
%	$$ \log(T)   < t_*/T $$ 
%	\textcolor{blue}{$\mathcal{L}(t \mapsto\Theta(t-t_*)) = \e^{-t_* s} s^{-1}$, also $\Re(s\eps(s)) = \Re s + \e^{-t_* \sigma}\cos(-t_* \omega)$% für $s = \sigma+i \omega, \sigma > 1$ }
%The bound from above is, for $\Re s > \sigma$, directly apparent since
%	\begin{equation*}
%	\abs{s^{-1}\left(\alpha_1+\alpha_2\e^{-t_* s}\right) } \le \sigma^{-1}(\alpha_1+\alpha_2).
%	\end{equation*}

	\smallskip
	%\noindent
	\eqref{eq:drude-kernel} The Laplace transform of the susceptibility kernel corresponding to the Drude model reads
\[	\widehat{\chi}_e(s) = \frac{\omega_D^2}{s(s+\gamma_D)}, \qquad\text{such that}\qquad
	\Re\bigl(\widehat{\chi}_e(s)s\bigr) = \frac{\omega_D^2}{\abs{s+\gamma_D}^2}\big(\gamma_D + \Re s\big) \ge 0.
\]	% the bound
%	\begin{equation*}
%	\frac{\omega_D^2}{\abs{s}\abs{s+\gamma_D}} \le \frac{\omega_D^2}{\sigma \gamma_D}.
%	\end{equation*}

	\smallskip
    \eqref{eq:lorentz-kernel} The Lorentz model is determined by
	\begin{equation*}
	\wchi_e(s) = \frac{\beta}{s(s +\alpha) +\omega}, 
	\quad\text{such  that}\quad
	\Re\big(\widehat{\chi}_e(s)s\big) = \frac{\beta}{\abs{s(s +\alpha) +\omega}^2} \Re\big(\abs{s}^2 s + \abs{s}^2 \alpha + s \omega\big)\ge 0.
	\end{equation*}
%	for which we obtain %the bounds of interest for $\Re s \ge \sigma$ via
%	\begin{align*}
%	\Re\big(s\,\widehat{\chi}_e(s)\big) = \frac{\beta}{\abs{s(s +\alpha) +\omega}^2} \Re\big(\abs{s}^2 s + \abs{s}^2 \alpha + s \omega\big)\ge 0, \qquad \Re s\ge 0.
	% \quad\text{ and }\quad\frac{\beta}{\abs{s(s +\alpha) +\omega}} \le \frac{\beta}{\sigma (\sigma+\alpha)},
	%\end{align*}
%	where the final estimate is obtained by estimating the denominator from below via
%	\begin{equation*}
%    \abs{s(s +\alpha) +\omega}
%    =
%\abs{s}\abs{s +\alpha+s^{-1}\omega}
%\ge
% \sigma\Re \left( s +\alpha+s^{-1}\omega \right)> \sigma(\sigma+\alpha) .
%	\end{equation*}
%	In the final inequality, we made use of $\Re s^{-1}>0$, which holds for $\Re s >0$.
	%and 
	%\begin{equation*}
	%\frac{\beta}{\abs{s(s +\alpha) +\omega}} \le \frac{\beta}{\omega}.
	%\end{equation*}

	\smallskip
\eqref{eq:fractional} The susceptibility kernel describing the fractional material law is characterized by the Laplace domain function with $\beta>0,\lambda>0$ and $ 0 < \eta < 2 $,
	\begin{equation*}
	\wchi_e(s) = \frac{\beta}{s^\eta +\lambda},
	\quad\text{ with }\quad
	\Re\bigl(\widehat{\chi}_e(s)s\bigr) = \frac{\beta\big(\Re(s\,\overline{s}^\eta) +\lambda\Re s\big)}{\abs{s^\eta+\lambda}^2}
	=
	\frac{\beta\big(\abs{s}^{2\eta}\big(\Re s^{1-\eta}+\lambda\Re s\big)}{\abs{s^\eta+\lambda}^2} \ge 0.
	\end{equation*}
%	We write $s = \abs{s}\e^{\mathrm{i} \theta}$ where $\theta \in (-\pi/2, \pi /2)$, then 
%    The passivity follows by estimating
%	\begin{align*}
%	\Re\big(s\,\widehat{\chi}_e(s)\big) &= \frac{\beta\big(\Re(s\,\overline{s}^\eta) +\lambda\Re s\big)}{\abs{s^\eta+\lambda}^2}
%	=
%	\frac{\beta\big(\abs{s}^{2\eta}\big(\Re s^{1-\eta}+\lambda\Re s\big)}{\abs{s^\eta+\lambda}^2} \ge 0.
%	%\\
%%	& > \frac{\beta \abs{s}^{1+\eta}}{\abs{s^\eta+\lambda}^2}\cos((1-\eta)\theta) \\
%%	& \ge 0,
%	\end{align*}
%    The above estimate crucially uses the restriction $0\le\eta \le 1$. 
%    
%    
%    The bound from above is obtained in the same way as for the Debye model. 
%     We similarly obtain 
%	\begin{align*} \frac{\beta}{\abs{s^\eta+\lambda}} \le \frac{\beta}{\lambda}.
%	\end{align*}
The bound \eqref{eq:bound-eps-mu} is obvious for each example.
\end{proof}

\subsection{Temporal convolution}
\label{subsec:Z}

\noindent Let $K(s)\colon X \rightarrow Y$, for $\Re s>0$, be an analytic family of bounded linear operators between two Hilbert spaces $X$ and $Y$. We assume that $K$ is polynomially bounded: there exists a real $\kappa$, and for every $\sigma >0$ there exists $M_\sigma <\infty$, such that
\begin{align}\label{eq:pol_bound}
\norm{K(s)}_{Y\leftarrow X}&\leq M_\sigma \abs{s}^\kappa, \quad\ \text{ Re } s \ge \sigma>0.
\end{align}
%In our examples we always have $\kappa\le 1$, and for convenience we assume this (inessential) bound of $\kappa$ in the following.
%As a key property, we further assume that $Z$ is of {\it positive type\/}:
%for every ${\sigma> \sigma_0\ge 0}$, there exists $c_\sigma>0$ such that
%\begin{equation}\label{eq:positive_type}
%\Re [ \phi, Z(s)\phi ] \ge c_\sigma\Re  s\, \bigl| s^{-1}\phi \bigr| _{\mVG}^2
%\quad \text{for all } \phi \in \mVG \text{  and }\text{Re }s\ge \sigma,
%\end{equation}
%where $[\cdot,\cdot]$ denotes the anti-duality between $\mVG$ and $\mVG'$, taken anti-linear in the first argument.
This bound  ensures that $K$ is the Laplace transform of a distribution of finite order of differentiation with support on the non-negative real half-line $t \ge 0$. For a function $g:[0,T]\to X$, which together with its extension by~$0$ to the negative real half-line is sufficiently regular, we use the Heaviside operational calculus notation
\begin{equation} \label{Heaviside}
K(\pt)g = (\mathcal{L}^{-1}K) * g
\end{equation}
for the temporal convolution of the inverse Laplace transform of $Z$ with~$g$. For the identity operator $\mathrm{Id}(s)=s$, we have $\mathrm{Id}(\pt)g = \pt g$, the time derivative of $g$.
For two such families of operators $K(s)$ and $L(s)$ mapping into compatible spaces, the associativity of convolution and the product rule of Laplace transforms yield the composition rule
\begin{equation}\label{comp-rule}
K(\pt)L(\pt)g = (KL)(\pt)g.
\end{equation}

%% noindent
For a Hilbert space $X$, we let $H^r(\R,X)$ be the Sobolev space of order $r$ of $X$-valued functions on $\R$,
and on finite intervals $(0, T )$
we let
$$
H_0^r(0,T;X) = \{g|_{(0,T)} \,:\, g \in H^r(\R,X)\ \text{ with }\ g = 0 \ \text{ on }\ (-\infty,0)\} ,
$$
where the
subscript 0 in ${H_0^r}$ only refers to the left end-point of the interval.
For integer $r\ge 0$, the norm $\| \pt^r g \|_{L^2(0,T;X)}$ is equivalent to the natural norm on $H_0^r(0,T;X)$.
The Plancherel formula yields the following bound \cite[Lemma 2.1]{L94}:
If $K(s)$ is bounded by \eqref{eq:pol_bound} in the half-plane $\text{Re }s > 0$, then $K(\pt)$ extends by density to a bounded linear operator $K(\pt)$ from $H^{r+\kappa}_0(0,T;X)$ to $H^r_0(0,T;Y)$ with the bound
\begin{equation}\label{sobolev-bound}
\| K(\pt) \|_{ H^{r}_0(0,T;Y) \leftarrow H^{r+\kappa}_0(0,T;X)} \le e M_{1/ T}
\end{equation}
for arbitrary real $r$. Here the bound on the right-hand side arises from the bound $e^{\sigma T} M_\sigma$ on choosing $\sigma=1/T$.
We note that for any integer $k\ge 0$ and $\alpha>1/2$, we have the continuous embedding $H^{k+\alpha}_0(0,T;X)\subset C^k([0,T];X)$.

%\begin{remark} Consider temporal kernels of the form
%\begin{equation*}
%\tilde{\eps}(t) = \int_0^\infty \e^{-\lambda t} \beta(\lambda) \dleb{\lambda},
%\end{equation*}
%a so-called completely monotonic kernel, with non negative function $\beta$, cf.~\cite{Wid41}, such that $\lambda \mapsto \beta(\lambda)\lambda^{-1}$ is integrable on the positive real line. Then the kernel
%\begin{equation*}
%\chi_e(t) = \eps_0\big(1+\tilde{\eps}(t)\big)
%\end{equation*}
%  satisfies the passivity condition \eqref{passive} and the bound \eqref{eq:bound-eps-mu}.
%\end{remark}
%
%\begin{proof}
%The Laplace transformed kernel is given by
%%
%\begin{align*}
%\tilde{\eps}(s) &= \int_0^\infty  \e^{-s t }\int_0^\infty \e^{-\lambda t }\beta(\lambda) \dleb{\lambda}\dleb{t}
%%\\
%%&= \int_0^\infty \int_0^\infty \e^{-s t } \e^{-\lambda t }\dleb{t}\,\beta(\lambda) \dleb{\lambda}\\
%%
%%&
%= \int_0^\infty \frac{1}{s+\lambda}\,\beta(\lambda) \dleb{\lambda},
%\end{align*}
%%
%such that, for $\Re s >0$,
%%
%\begin{align*}
%\Re \tilde{\eps}(s) s&= \int_0^\infty \Re \Bigl( \frac{s}{s+\lambda}\Bigr)\,\beta(\lambda) \dleb{\lambda} \geq 0,
%\end{align*}
%%
%and
%%
%\begin{align*}
%\abs{\tilde{\eps}(s) s} &\leq
%%  \int_0^\infty  \frac{\abs{s}}{\abs{s+\lambda}}\,\abs{\beta(\lambda)} \dleb{\lambda} 
%%= 
%%\abs{s}\int_0^\infty  \frac{\abs{\beta(\lambda)}}{\abs{s+\lambda}}\, \dleb{\lambda} 
%%\le 
%\abs{s}\int_0^\infty  \frac{\beta(\lambda)}{\abs{\lambda}}\, \dleb{\lambda} = C_\eps \abs{s},
%\end{align*}
%which gives the claim.
%\end{proof}

\subsection{The tangential trace and the trace space ${\textbf{\textit{X}}}_\Gamma$}
Let $\Omega$ be a bounded Lipschitz domain in $\R^3$ with boundary surface $\Gamma=\partial\Omega$ or the complement of the closure of such a domain. For a continuous vector field $\u:\overline\Omega\to\C^3$,
we define the {tangential trace}
\begin{align*}
\gaT \u= \u|_\Gamma \times \normal \qquad\text{on }\Gamma,
\end{align*}
where $\normal$ denotes the unit surface normal  pointing into the exterior domain.  We note that the tangential component of $\u|_\Gamma$ is $\u_T=(\boldsymb I-\normal\normal^\top)\u|_\Gamma = - (\gaT \u)\times \normal$.

By the version of Green's formula for the $\curl$ operator, we have for sufficiently regular vector fields $\u , \v : \overline\Omega\to\C^3$ that
\begin{equation}\label{eq: Green}
\int_\Omega  \bigl( \curl \u \cdot \v -  \u \cdot \curl \v \bigr)\textrm{d} \x = \int_\Gamma  (\gaT \u \times \normal) \cdot \gaT \v  \,\textrm{d}\sigma,
\end{equation}
where the dot $\cdot$ stands for the Euclidean inner product on $\C^3$, i.e., $\a \cdot \b = \overline{\a}^\top \b$ for $\a,\b\in\C^3$. The right-hand side in
this formula defines a skew-hermitian sesquilinear form on continuous tangential vector fields on the boundary, say $\bphi,\bpsi:\Gamma\to\C^3$, which we write as
\begin{equation}\label{skew}
[\bphi,\bpsi]_\Gamma = \int_\Gamma (\bphi\times\normal)\cdot \bpsi \, \textrm{d}\sigma.
\end{equation}
As it was shown in \cite{AV96} for smooth domains and extended by \cite{BCS02} for Lipschitz domains
(see also the survey in \cite[Sect.\,2.2]{BH03}), the trace operator $\gaT$ can be extended to a surjective bounded linear operator from the space that appears naturally for Maxwell's equations,
$
\H(\curl,\Omega) = \{\v \in \bL^2(\Omega)\,:\, \curl \v \in \bL^2(\Omega) \},
$
to the
$$
\text{
	{trace space}: a Hilbert space denoted $\mXG$, with norm $\|\cdot\|_\mXG$.
}
$$
This space is characterized as the tangential subspace of the Sobolev space $\H^{-1/2}(\Gamma)$  with surface divergence in $H^{-1/2}(\Gamma)$ (see the papers cited above for the precise formulation). It has the property that
the pairing $[ \cdot,\cdot ]_\Gamma$ can be extended to a non-degenerate continuous sesquilinear form on $\mXG\times \mXG$. With this pairing the space $\mXG$ becomes its own dual.

\section{A time-harmonic transmission problem }\label{sec:time-harmonic-transmission}

For the derivation of the parameter-dependent operators and representation formulas we write in this section $\varepsilon(s)$ and $\mu(s)$ either for $\varepsilon^{+}(s)$ and $\mu^{+}(s)$ or for
$\varepsilon^{-}(s)$ and $\mu^{-}(s)$.
%change the notation from $\varepsilon^{\pm}(s)$ and $\mu^{\pm}(s)$ to $\varepsilon(s)$ and $\mu(s)$.
The single and double layer potentials and the Calder\'{o}n operator are defined for the same material parameters in the inner and outer domain. In the derivation of the boundary integral equation, this suits the situation, since the solutions of Maxwell's equations are extended to $\R^3 \setminus \Gamma$, by setting them to zero on either the inner or outer domain. For notational simplicity, in the proofs we omit the frequency-dependence in the notation.

Formally applying the Laplace transform to Maxwell's equation and inserting the material law \eqref{eq:material law} in \eqref{eq:MW-complete} yields the time-harmonic Maxwell's equations
\begin{alignat}{2}
\label{TH-MW12}
\begin{aligned}
\varepsilon(s)s  \widehat{\E}-\curl \widehat{\B} &= 0
%	\widehat{\J}
\\
\mu(s)s  \widehat{\B} + \curl \widehat{\E} &=0  
\end{aligned}\quad\quad \text{in} \ \mathbb R^3\setminus \Gamma
\end{alignat}
with the complex-valued analytic functions $\varepsilon$ and $\mu$ that satisfy the passivity condition
\eqref{passive}.
%, connected with the Laplace transforms of the susceptibility kernels through
%\begin{equation}
%\varepsilon(s)= \varepsilon_0 (1+\widehat \chi_e (s)),\quad 	\mu(s)= \mu_0 (1+\widehat \chi_m (s)).
%\end{equation}

\subsection{Potential operators and representation formulas}
%Consider Green's formula
%
%\begin{equation}\label{eq: Green}
%\int_\Omega  \bigl( \curl \bu \cdot \bv -  \bu \cdot \curl \bv \bigr)\, \textrm{d} \bx = \int_\Gamma  (\gaT \bu \times \normal) \cdot \gaT \bv  \,\textrm{d}\bx,
%\end{equation}
%The following terminology in this subsection are not original results, but rather standard notions surrounding the time-harmonic Maxwell's equations. 
%
%More details can be found in the survey \cite{BH03} and the classical reference \cite{N01}. 

The fundamental solution of the time-harmonic Maxwell's equations with $\varepsilon=\mu=1$ reads
%For $\Re s>0$ and $\x\in\mathbb{R}^3 \setminus \{0\}$, it is given by
\begin{align*}
%	G(s,z)= \dfrac{e^{-s\sqrt{\varepsilon \mu }\norm{z}}}{4\pi\norm{z}}.
G(s,\x)= \dfrac{e^{-s\abs{\x}}}{4\pi\abs{\x}}, \qquad \Re s>0,\  \x\in\mathbb{R}^3 \setminus \{0\}.
\end{align*}

%\begin{align*}
%%	G(s,z)= \dfrac{e^{-s\sqrt{\varepsilon \mu }\norm{z}}}{4\pi\norm{z}}.
%G(s,\x)= \dfrac{e^{-s \sqrt{\varepsilon\mu}\abs{\x}}}{4\pi\abs{\x}}, \qquad \Re s>0,\  \x\in\mathbb{R}^3 \setminus \{0\}.
%\end{align*}
The \emph{electromagnetic single layer potential} operator is denoted by $\mathbfcal{S}(s)$. Applied to a complex-valued boundary function $\bvar$ of sufficient regularity for the expressions to be finite, and evaluated at a point $\x\in \mathbb{R}^3 \setminus \Gamma$ away from the boundary, it reads
\begin{align*}
\mathbfcal{S}(s)\bvar(\x)= -s\int_\Gamma G(s,\x-\y)\bvar(\y)\text{d}\y +  s^{-1}\nabla \int_\Gamma G(s,\x-\y) \divG \bvar(\y) \text{d}\y.
\end{align*}
%\begin{align*}
%\mathbfcal{S}(s)\bvar(\x)= -s\sqrt{\varepsilon\mu}\int_\Gamma G(s,\x-\y)\bvar(\y)\text{d}\y + \frac{1}{s\sqrt{\varepsilon\mu}} \nabla \int_\Gamma G(s,\x-\y) \divG \bvar(\y) \text{d}\y.
%\end{align*}
The \emph{electromagnetic double layer potential} operator is denoted by $\mathbfcal{D}(s)$ and is given in the same context by
%, for $\x\in \mathbb{R}^3 \setminus \Gamma$,
\begin{align*}
\mathbfcal{D}(s)\bvar(\x) = \curl\int_\Gamma G(s,\x-\y)\bvar(\y)\text{d}\y.
\end{align*}
By construction, the potential operators satisfy the relations
\begin{equation}\label{SD-MW}
s \mathbfcal{S}(s) - \curl \circ\, \mathbfcal{D}(s) = 0, \qquad s \mathbfcal{D}(s) + \curl \circ\, \mathbfcal{S}(s) =0.
\end{equation}
%\begin{equation}\label{SD-MW}
%s \sqrt{\varepsilon\mu} \mathbfcal{S}(s) - \curl \circ\, \mathbfcal{D}(s) = 0, \qquad s\sqrt{\varepsilon\mu} \mathbfcal{D}(s) + \curl \circ\, \mathbfcal{S}(s) =0.
%\end{equation}
%\ \\ \ \\For any boundary function $\bvar$ of appropriate regularity, the fields ${\widehat \E = \mathbfcal{S}(s)\bvar}$ and ${\widehat \H = \mathbfcal{D}(s)\bvar}$ solve the time-harmonic Maxwell's equations \eqref{TH-MW1}--\eqref{TH-MW2} on ${\mathbb{R}^3 \setminus \Gamma}$. Likewise, due to the inherent anti-symmetry of Maxwell's equations, the same statement holds for the fields ${\widehat \E = \mathbfcal{D}(s)\bvar}$ and ${\widehat \H = -\mathbfcal{S}(s)\bvar}$. Note that without the appropriate rescaling and assumptions on $\varepsilon$ and $\mu$, as discussed in the introduction, these constants would enter into the formulas above.
This section relies heavily on electromagnetic transmission problems, formulated on $\mathbb{R}^3\setminus \Gamma$. Jumps and averages for the tangential traces are defined by
\begin{alignat*}{2}
\jmp{\gamma_T} &=  \gamma_T^+-\gamma_T^- , \quad \quad
\avg{\gamma_T}  &&=  \tfrac12\left(\gamma_T^++\gamma_T^-\right).	
\end{alignat*}
The composition of the jumps with the potential operators reveals the {\it jump relations}
\begin{equation}\label{jump-rel}
\jmp{\gaT}\circ \mathbfcal{S}(s) =0, \qquad \jmp{\gaT}\circ \mathbfcal{D}(s) = -  \bId  .
\end{equation}
For general $\varepsilon,\mu$, we use the potential operators
\begin{equation} \label{SD-eps-mu}
	 \mathbfcal{S}\epsmu(s) = 	 \mathbfcal{S}\big(s\sqrt{\varepsilon(s)\mu(s)}\big) ,\quad 	 \mathbfcal{D}\epsmu(s) = 	 \mathbfcal{D}\big(s\sqrt{\varepsilon(s)\mu(s)}\big).
\end{equation}
The identities \eqref{SD-MW} and the jump relations \eqref{jump-rel} imply that any sufficiently regular boundary densities $(\bwvarphi,\bwpsi)$ % from the trace space $\mXG\times \mXG$ 
are associated with electromagnetic fields $(\widehat{\E},\widehat{\H})$ by 
\begin{align}
\widehat{\E}&=-\frac{\sqrt{\mu(s)}}{\sqrt{\varepsilon(s)}}	 \mathbfcal{S}\epsmu(s)\bwvarphi + \mathbfcal{D}\epsmu(s)\bwpsi,\label{eq:time-harmonic-kirchhoff-E3}
\\
 \widehat{\H}&= - \mathbfcal{D}\epsmu(s)\bwvarphi -\frac{\sqrt{\varepsilon(s)}}{\sqrt{\mu(s)}}	 \mathbfcal{S}\epsmu(s)\,\bwpsi, \label{eq:time-harmonic-kirchhoff-H3}
\end{align}
which solve the transmission problem %(assuming $\eps\mu=1$)
\begin{alignat}{2}
\, \varepsilon(s)s\widehat{\E}-\curl \widehat{\B} &=0 %\widehat{J}
\quad\quad\quad &&\text{in} \quad \mathbb{R}^3\setminus\Gamma ,\label{eq:transmis-1-thMW}
\\
\mu(s)s\widehat{\B} + \curl \widehat{\E} &=0 \quad &&\text{in} \quad \mathbb{R}^3\setminus\Gamma ,\label{eq:transmis-2-thMW} \\
\,\jmp{\gaT}\widehat{\B}&= \bwvarphi\,,\label{eq:transmis-3-thMW} &&\\
-\jmp{\gaT} \widehat \E&=\bwpsi \,.\label{eq:transmis-4-thMW} &&
\end{alignat}
Up to this point, this section was restricted to the presentation of established operators and identities, which hold for boundary densities of sufficient regularity. The next subsection provides bounds in terms of the appropriate norms, which in particular gives a rigorous setting for the previously defined operators. Before that, we turn to some useful estimates of the terms in formulas~\eqref{SD-eps-mu} 
and~\eqref{eq:time-harmonic-kirchhoff-E3}--\eqref{eq:time-harmonic-kirchhoff-H3}.

The following lemma shows that $ \mathbfcal{S}\epsmu(s)$ and $ \mathbfcal{D}\epsmu(s)$ behave well for $\Re s >0$.

	\begin{lemma}\label{lem:eps-mu-s}
		%Under Assumption~\ref{strongly-passive}, 
		Under the passivity condition \eqref{passive},
		the argument appearing in the definition of the potential operators $\mathbfcal S_{\varepsilon,\mu}(s)$ and $\mathbfcal 
		D_{\varepsilon,\mu}(s)$ has positive real part: 
		\begin{align}\label{pos-sqrt-eps-mu-s}
			\Re \bigl(s\sqrt{\varepsilon(s)\mu(s)}\bigr) > 0 \quad\ \text{for }\ \Re s >0.
		\end{align}
		Under the strong passivity condition \eqref{strongly-passive}, we have with $c^{-1}=\sqrt{\eps_0\mu_0}$
		\begin{align}\label{strong-pos-sqrt-eps-mu-s}
			\Re \bigl(s\sqrt{\varepsilon(s)\mu(s)}\bigr) \ge c^{-1} \, \Re s  \quad\ \text{for }\ \Re s >0.
		\end{align}
\end{lemma}
\begin{proof}
We write  $\varepsilon(s) s=\abs{\varepsilon(s)s}e^{i\varphi_\varepsilon}$ and $\mu(s)s =\abs{\mu(s)s}e^{i\varphi_\mu}$, with $\varphi_\mu,\varphi_{\varepsilon}\in (-\pi/2,\pi/2)$ due to the positivity \eqref{passive}. We then have
$$
\Re \bigl(s\sqrt{\varepsilon(s)\mu(s)}\bigr)
=
\abs{\varepsilon(s)s}^{1/2}\abs{\mu(s)s}^{1/2}\Re  e^{i(\varphi_\mu+\varphi_\varepsilon)/2},
$$
which is positive since $\Re  e^{i(\varphi_\mu+\varphi_\varepsilon)/2} = \cos((\varphi_\mu+\varphi_\varepsilon)/2)>0$.
The inequality \eqref{strong-pos-sqrt-eps-mu-s} follows from the general inequality, for $a,b\in\C$ with $\Re a \ge 0$ and $\Re b\ge 0$,
$$
\Re \sqrt{ab} \ge \sqrt{\Re a \cdot \Re b}.
$$
This inequality is proved using polar coordinates for $a = |a| \, e^{i\alpha}$ and $b=|b|e^{i\beta}$ and the inequalities
$$
\cos \bigl(\tfrac12(\alpha + \beta)\bigr) \ge \tfrac12 \bigl(\cos\alpha + \cos\beta\bigr) \ge \sqrt{\cos\alpha \cdot \cos\beta},
$$
where the first inequality results from the concavity of the cosine on $[-\pi/2,\pi/2]$ and the second inequality is the arithmetic-geometric mean inequality.
\end{proof}

In view of \eqref{eq:time-harmonic-kirchhoff-E3}--\eqref{eq:time-harmonic-kirchhoff-H3} we further note that under the strong passivity condition~\eqref{strongly-passive} and the bound \eqref{eq:bound-eps-mu} we have the bounds, for $\Re s \ge \sigma >0$,
\begin{equation}\label{eps-mu-factor-bounds}
\left|\frac{\sqrt{\mu(s)}}{\sqrt{\varepsilon(s)}}\right| = \left|\frac{\sqrt{\mu(s)s}}{\sqrt{\varepsilon(s)s}}\right| \le 
\biggl( \frac{\mu_0\, M_\sigma}{\varepsilon_0} \biggr)^{1/2} \frac{|s|^{1/2}}{(\Re s)^{1/2}}
\quad\text{and}\quad 
\left|\frac{\sqrt{\varepsilon(s)}}{\sqrt{\mu(s)}}\right| %= \left|\frac{\sqrt{\varepsilon(s)s}}{\sqrt{\mu(s)s}}\right| 
\le 
\biggl( \frac{\varepsilon_0 \,M_\sigma}{\mu_0} \biggr)^{1/2} \frac{|s|^{1/2}}{(\Re s)^{1/2}}.
\end{equation}

%
%Alternatively, for any given boundary densities $(\bwvarphi,\bwpsi)$ from the trace space $\mXG\times \mXG$ are associated to electromagnetic fields
%
%\begin{align}
%\sqrt{\varepsilon}\widehat{\E}&=-\mathbfcal{S}(s)\bwvarphi + \mathbfcal{D}(s)\bwpsi,\label{eq:time-harmonic-kirchhoff-E3-2}
%\\
%\sqrt{\mu} \widehat{\H}&= - \mathbfcal{D}(s)\bwvarphi -\mathbfcal{S}(s)\,\bwpsi, \label{eq:time-harmonic-kirchhoff-H3-2}
%\end{align}
%which are fields that solve the transmission problem %(assuming $\eps\mu=1$)
%\begin{alignat}{2}
%\, s\varepsilon\widehat{\E}-\curl \widehat{\B} &=0 %\widehat{J}
%\quad\quad\quad &&\text{in} \quad \mathbb{R}^3\setminus\Gamma ,\label{eq:transmis-1-thMW-2}
%\\
%s\mu\widehat{\B} + \curl \widehat{\E} &=0 \quad &&\text{in} \quad \mathbb{R}^3\setminus\Gamma ,\label{eq:transmis-2-thMW-2} \\
%\,\jmp{\gaT}\sqrt{\mu}\widehat{\B}&= \bwvarphi\,,\label{eq:transmis-3-thMW-2} &&\\
%-\jmp{\gaT}\sqrt{\varepsilon} \widehat \E&=\bwpsi \,.\label{eq:transmis-4-thMW-2} &&
%\end{alignat}

\subsection{ Transmission problems and boundary operators}
The right-hand side of the representation formula, namely the  operator associated to the linear map
$(\bwvarphi,\bwpsi)\mapsto (\widehat{\E},\widehat{\B})$, extends by density to a bounded linear operator from the trace space $\mXG^2$ to
$\H(\curl,\Omega)^2$. The following lemma proves this and further provides an $s$-explicit bound. A related result can be found in \cite[Lemma~6.4]{ChanMonk2015}.

\begin{lemma}\label{lem:transmission}
	Let $(\bwvarphi,\bwpsi)\in \mXG^2$ be some complex-valued boundary functions in the trace space. There exist time-harmonic electromagnetic fields $(\widehat \E,\widehat \B)$, that are defined by the representation formulas \eqref{eq:time-harmonic-kirchhoff-E3}--\eqref{eq:time-harmonic-kirchhoff-H3}, which solve the transmission problem \eqref{eq:transmis-1-thMW}--\eqref{eq:transmis-4-thMW} for $\Re s >0$ and are bounded by
	\begin{equation}
	\label{eq:transmis-bound-MW}
	\left\|\begin{pmatrix} \widehat{\E} \\ \widehat{\B} \end{pmatrix} \right\|_{\H(\curl,\mathbb{R}^3\setminus \Gamma)^2}
	\le  \Ctrace 	\max\left(\frac{|\varepsilon(s) s|^2+1}{\Re \varepsilon(s) s}\,,\,\frac{|\mu(s) s|^2+1}{\Re \mu(s) s}\right)
	\left\|\begin{pmatrix} \bwvarphi \\ \bwpsi \end{pmatrix} \right\|_{\mXG^2},
	\end{equation}
	%\begin{align*}
	%&\Bigl(\big\| \widehat{\E} \big\|_{\H(\curl,\mathbb{R}^3\setminus \Gamma)}^2 +
	%         \big\| \mu\widehat{\H} \big\|_{\H(\curl,\mathbb{R}^3\setminus \Gamma)}^2	\Bigr)^{1/2}
	%% \\
	%% &\qquad\qquad
	% \le
	%  \Ctrace \dfrac{\abs{s}^{2}+1}{\Re s}  \Bigl( \| \wvarphi \|_\mXG^2 + \| \wpsi \|_\mXG^2 \Bigr)^{1/2},
	%\end{align*}
	where the constant ${\Ctrace=\| \{\gaT\} \|_{\mXG \leftarrow \H(\curl,\mathbb{R}^3\setminus\Gamma)}}$ is the operator norm of the tangential trace operator.
\end{lemma}
\begin{proof}
	Throughout the proof, we omit the frequency variable $s$ in the material parameters $\varepsilon(s)$ and $\mu(s)$.
	Green's formula in combination with the time-harmonic Maxwell's equations reads
	\begin{align}
	\nonumber
	\pm\left[\gaT^\pm \widehat \B, \gaT^\pm \widehat \E\right]_\Gamma
	&=
	\int_{\Omega^\pm} \bigl(   \curl \widehat\B\cdot \widehat \E - \widehat \B \cdot\curl \widehat \E  \bigr)\,\textrm{d} \bx
	\\
	\label{Green-EH}
	&=
	\int_{\Omega^\pm}  \bigl(\bar \varepsilon \bar s  \big| \widehat \E \big|^2+  \mu s  \big| \widehat \B \big|^2\bigr)\, \textrm{d} \bx.
	\end{align}	
	Recall that $\Omega^-$ and $\Omega^+$ refer to the interior and exterior domain, respectively. The conjugation of the Laplace parameter in the first summand stems from the anti-linearity of the inner product, which has been defined via ${\a \cdot \b = \overline{\a}^\top \b}$ on $\mathbb{C}^3$. Summation of these two terms yields the identity
	%which after multiplying with $\mu$ and using $\eps\mu=1$ gives
	\begin{equation}\label{Green-EH-trans}
	\I\coloneqq \int_{\mathbb{R}^3\setminus \Gamma} \bar \varepsilon \bar s \big| \widehat \E \big|^2+  \mu s  \big|  \widehat \B \big|^2 \textrm{d} \bx
	=  \left[\gaT^+ \widehat \B, \gaT^+ \widehat \E\right]_\Gamma
	- \left[\gaT^- \widehat \B,  \gaT^- \widehat \E\right]_\Gamma .
	\end{equation}
	Any part of the time-harmonic electromagnetic fields can always be rewritten in terms of each others $\curl$, by inserting   \eqref{eq:transmis-1-thMW} and \eqref{eq:transmis-2-thMW} respectively. Using the separation $\I = (1-\theta)\I + \theta \I$ and inserting the time-harmonic Maxwell problem in the second summand reformulates the left-hand side to the expression
	\begin{align*}
	%	&\int_{\mathbb{R}^3\setminus \Gamma}  \Bigl( s  \big|\widehat{\E}\big|^2
	%	+\bar s  \big|\mu \widehat{\H}\big|^2 \Bigr)  \mathrm{d}x
	%	\\ &
	\I &=
	\int_{\mathbb{R}^3\setminus \Gamma} \Big( (1-\theta_1) \bar \varepsilon \bar s \big|\widehat{\E}\big|^2
	+\theta_2 \mu s \big|(\mu s)^{-1}\curl \widehat{\E}\big|^2
	\\
	& \qquad \qquad \quad +
	(1-\theta_2)\mu s  \big| \widehat{\B}\big|^2
	+\theta_1 \bar \varepsilon \bar s  \big| (\varepsilon s)^{-1}  \curl \widehat{\B}\big|^2  \Big)
	\mathrm{d}\bx .
	\end{align*}

	Taking the real part on both sides slightly simplifies the right-hand side to
	\begin{align*}
	\Re \I &= 
	%\min(\Re \varepsilon s,\Re \mu s) 
	\int_{\mathbb{R}^3\setminus \Gamma} \Big( (1-\theta_1)  \Re \varepsilon s\big|\widehat{\E}\big|^2
	+\theta_2 \abs{\mu s}^{-2}\Re \mu s \big|\curl \widehat{\E}\big|^2
	\\
	& \qquad \qquad +
	(1-\theta_2) \Re \mu s \big| \widehat{\B}\big|^2
	+\theta_1 \abs{\varepsilon s}^{-2}\Re \varepsilon s \big|\curl \widehat{\B}\big|^2  \Big)
	\mathrm{d}\bx .
	\end{align*}
	%	To optimize the right-hand side in terms of the norm associated to $\H(\curl,\mathbb R^3\setminus \Gamma)$, 
	The parameters $(\theta_1,\theta_2)$ are free and chosen in such a way that the preceding factors of the summands agree, which is achieved by setting $1-\theta_1=\theta_1 |\varepsilon s|^{-2}$ and $1-\theta_2=\theta_2 |\mu s|^{-2}$. Rearranging this requirement leads to the choice of ${\theta_1=1/(1+|\varepsilon s|^{-2})}$and ${\theta_2=1/(1+|\mu s|^{-2})}$. Inserting these particular choices of $\theta_1$ and $\theta_2$ yields the estimate
	\begin{equation}\label{ReI-1}
	\Re \I \ge
	\min\left(\frac{\Re \varepsilon s}{|\varepsilon s|^2+1}\,,\,\frac{\Re \mu s}{|\mu s|^2+1}\right) \Bigl( \| \widehat \E \|_{\H(\curl,\R^3\setminus\Gamma)}^2 + \|  \widehat \B \|_{\H(\curl,\R^3\setminus\Gamma)}^2 \Bigr).
	\end{equation}
    The stated result follows now from following the proof of \cite[Lemma 3.1]{NKL22} on from the identity (3.14). To keep the proof self-contained, we conclude with the arguments given there.
 
 %The real part of the term $\I$ is bounded from below by the left-hand side of the stated bound \eqref{eq:transmis-bound-MW}.  
	The real part of $\I$ is, due to the right-hand side of \eqref{Green-EH-trans}, also characterized by
	$$%\begin{equation}\label{ReI-2}
	\Re \I =
	\Re \Bigl( \left[\gaT^+ \widehat \B, \gaT^+ \widehat \E\right]_\Gamma
	- \left[\gaT^- \widehat \B,  \gaT^- \widehat \E\right]_\Gamma \Bigr).
	$$%\end{equation}
	Rewriting the right-hand side in terms of jumps and averages bysumming several mixed terms and using the transmission conditions \eqref{eq:transmis-3-thMW}--\eqref{eq:transmis-4-thMW} yields
	\begin{align} \label{ReI-2}
	\Re \I
	&= \Re \Bigl( \left[ \jmp{\gaT }\widehat \H, \avg{\gaT}\widehat{\E} \right]_\Gamma  + \left[ -\jmp{\gaT} \widehat \E, \avg{\gaT}\widehat{\H}\right]_\Gamma \Bigr)
	\\ \nonumber
	&=
	\Re \Bigl( \left[ \bwvarphi, \avg{\gaT}\widehat{\E} \right]_\Gamma
	+ \left[ \bwpsi, \avg{\gaT}\widehat{\B}\right]_\Gamma \Bigr).
	\end{align}
	The self-duality of $\mXG$ implies a Cauchy--Schwarz type inequality with the corresponding norm and the duality pairing $[\cdot,\cdot]_\Gamma$. Combined with the Cauchy--Schwarz inequality on $\R^2$, this yields %and the bound $C_\Gamma$ of $\avg{\gaT}:\H(curl,\R^3\setminus\Gamma)\to\mXG$
	%to estimate
	\begin{align*}
	\Re \I & \le \| \bwvarphi \|_{\mXG}  \, \| \avg{\gaT}\widehat{\E} \|_{\mXG} +
	\| \bwpsi \|_{\mXG}  \, \| \avg{\gaT}\widehat{\B} \|_{\mXG}
	= \begin{pmatrix}
	\| \bwvarphi \|_{\mXG} \\ \| \bwpsi \|_{\mXG} 
	\end{pmatrix} \cdot \begin{pmatrix}
	\| \avg{\gaT}\widehat{\E} \|_{\mXG}  \\ \| \avg{\gaT}\widehat{\B} \|_{\mXG}
	\end{pmatrix}
	\\
	&\le \Bigl( \| \bwvarphi \|_{\mXG} ^2 + \| \bwpsi \|_{\mXG} ^2 \Bigr)^{1/2} %\, C_\Gamma
	\Bigl( \| \avg{\gaT}\widehat{\E} \|_{\mXG}^2 + \|\avg{\gaT}  \widehat{\B} \|_{\mXG}^2 \Bigr)^{1/2}.
	\end{align*}
	To estimate the second factor of the above expression, we intend to use the bound of the tangential trace $\avg{\gaT}: \H(\curl,\R^3\setminus\Gamma) \to \mXG$.
	The time-harmonic electromagnetic fields $\widehat \E$ and $\widehat \B$ are in the local Sobolev space $\H_{\mathrm{loc}}(\curl,\R^3\setminus\Gamma)$ (c.f. \cite{BH03}). Moreover, the tangential trace $\avg{\gaT}$ extends to a bounded operator from $\H(\curl,\Omega_\Gamma)$ to $\mXG$, where $\Omega_\Gamma$ is a bounded domain large enough to contain the boundary $\Gamma$. Hence, the left-hand side $\Re \I$ is bounded and the electromagnetic fields are in the global Sobolev space $\H(\curl,\R^3\setminus\Gamma)$. With the operator norm of the tangential average ${\Ctrace=\| \{\gaT\} \|_{\mXG \leftarrow \H(\curl,\mathbb{R}^3\setminus\Gamma)}}$, the right-hand side is therefore bounded via
	$$
	\Re \I \le C_\Gamma \Bigl( \| \bwvarphi \|_{\mXG} ^2 + \| \bwpsi \|_{\mXG} ^2 \Bigr)^{1/2}
	\Bigl( \| \widehat \E \|_{\H(\curl,\R^3\setminus\Gamma)}^2 + \|  \widehat \B \|_{\H(\curl,\R^3\setminus\Gamma)}^2 \Bigr)^{1/2}.
	$$
	Inserting \eqref{ReI-1} on the left-hand side and dividing through the second factor on the right-hand side yields the stated bound.
\end{proof}

%\begin{proof} 
%	\begin{align*}
%	\left\|\begin{pmatrix} \widehat{\E} \\ \mu \widehat{\B} \end{pmatrix} \right\|_{\H(\curl,\mathbb{R}^3\setminus \Gamma)^2}^2
%	&= \|\widehat{\E}\|_{\H(\curl,\mathbb{R}^3\setminus \Gamma)^2}^2 +\|\mu \widehat{\B}\|_{\H(\curl,\mathbb{R}^3\setminus \Gamma)^2}^2\\
%	&= \left\|- \mathbfcal{S}\epsmu(s)\left(\frac{1}{\sqrt{\varepsilon\mu}}	\bwvarphi\right) + \mathbfcal{D}\epsmu(s)\bwpsi\right\|_{\H(\curl,\mathbb{R}^3\setminus \Gamma)^2}^2\\
%	&\quad +\left\|- \mathbfcal{D}\epsmu(s)\bwvarphi -	 \mathbfcal{S}\epsmu(s)\,\big(\sqrt{\varepsilon\mu}\bwpsi\big)\right\|_{\H(\curl,\mathbb{R}^3\setminus \Gamma)^2}^2\\
%	&\le \left(\Ctrace \, \dfrac{\abs{s\sqepsmu}^{2}+1}{\Re s\sqepsmu}\right)^2 \, \left\|\begin{pmatrix} \frac{1}{\sqrt{\varepsilon\mu}} \bwvarphi \\ \bwpsi \end{pmatrix} \right\|_{\mXG^2}^2\\
%	&\quad + \left(\Ctrace \, \dfrac{\abs{s\sqepsmu}^{2}+1}{\Re s\sqepsmu}\right)^2 \, \left\|\begin{pmatrix}  \bwvarphi \\ \sqrt{\varepsilon\mu}\,\bwpsi \end{pmatrix} \right\|_{\mXG^2}^2\\
%%
%%	
%%	\le  \Ctrace \, \dfrac{\abs{s\sqepsmu}^{2}+1}{\Re s\sqepsmu}\,
%%	\left\|\begin{pmatrix} \bwvarphi \\ \bwpsi \end{pmatrix} \right\|_{\mXG^2},
%	\end{align*}
%		\begin{align*}
%		\left\|\begin{pmatrix} \widehat{\E} \\ \mu \widehat{\B} \end{pmatrix} \right\|_{\H(\curl,\mathbb{R}^3\setminus \Gamma)^2}
%		&\le \Ctrace \,\sqrt{2}\max\left\{\sqepsmu, \frac{1}{\sqepsmu}\right\} \dfrac{\abs{s\sqepsmu}^{2}+1}{\Re s\sqepsmu} \, \left\|\begin{pmatrix}\bwvarphi \\ \bwpsi \end{pmatrix} \right\|_{\mXG^2}
%		\end{align*}
%\end{proof}

In both the time-dependent and time-harmonic situation, our approach consists of determining the tangential traces of the Maxwell solutions by the respective boundary integral equation, and inserting these into the representation formulas to obtain the electromagnetic fields. 
In this situation, the boundary densities reduce to the tangential traces of the interior and exterior fields respectively, which is a setting that enables an improvement of the bound described in Lemma~\ref{lem:transmission}. The following Lemma gives these improved bounds.

\begin{lemma}\label{lem:transmission-improved}
In the situation of Lemma~\ref{lem:transmission} further assume that the interior (exterior) tangential traces of $\widehat{\E}$ and $\widehat{\B}$ are indentically $0$, which implies $\gamma_T^-\widehat{\B} = \bwvarphi$ ($-\gamma^+_T\widehat{\B} = \bwvarphi$) and $-\gamma^-_T \widehat{\E} = \bwpsi$ ($\gamma^+_T \widehat{\E} = \bwpsi$). Then, the bound of Lemma~\ref{lem:transmission} improves to 
	\begin{equation}
\label{eq:transmis-bound-improved}
\left\|\begin{pmatrix} \widehat{\E} \\ \widehat{\B} \end{pmatrix} \right\|_{\H(\curl,\Omega^\pm)^2}
\le  \frac{1}{\sqrt{2}}	\left(\max\left(\frac{|\varepsilon(s) s|^2+1}{\Re \varepsilon(s) s}\,,\,\frac{|\mu(s) s|^2+1}{\Re \mu(s) s}\right)\right)^{1/2}
\left\|\begin{pmatrix} \bwvarphi \\ \bwpsi \end{pmatrix} \right\|_{\mXG^2}.
\end{equation}
Furthermore, we have the $L^2$-bound
\begin{equation*}
\left\|\begin{pmatrix} \widehat{\E} \\ \widehat{\B} \end{pmatrix} \right\|_{L^2(\Omega^\pm)^2} \le \frac{1}{\sqrt{2}}\left( \max \left(\frac{1}{\Re \eps(s)s},\frac{1}{\Re \mu (s)s}\right)\right)^{1/2} \left\|\begin{pmatrix} \bwvarphi \\ \bwpsi \end{pmatrix} \right\|_{\mXG^2}.
\end{equation*}
\end{lemma}

\begin{proof}
The proof of the $\H(\curl, \Omega^\pm)$ bound is identical to that of Lemma~\ref{lem:transmission} down to \eqref{ReI-2}, which now implies the bound $\Re I \le \frac{1}{2}\big(\norm{\bwvarphi}_{\mXG}^2+ \|\bwpsi\|_\mXG^2\big)$ and yields the stated result. The proof of the $L^2$-bound is even simpler, working directly with \eqref{Green-EH-trans} instead of \eqref{ReI-1}.
\end{proof}

\subsection{Time-harmonic boundary integral operators and the Calder\'{o}n operator}

The composition of the tangential averages with the potential operators defines the electromagnetic \emph{single and double layer boundary operators}, which operate on the trace space $\mXG$ and are defined as
$$
\V\epsmu(s)=\avg{\gaT}\circ \mathbfcal{S}\epsmu(s), \qquad \K\epsmu(s)=\avg{\gaT}\circ \mathbfcal{D}\epsmu(s).
$$
%As was shown in \cite{BH03}, these boundary integral operators are continuous maps on the trace space $\mXG$, i.e., for $\Re\,s>0$,
%\begin{align*}
%V(s):\mXG\rightarrow\mXG, \quad \quad K(s) :\mXG \rightarrow \mXG.
%\end{align*}
%Although $s$-explicit bounds on $V(s)$ and $K(s)$ have been shown in \cite[Theorem 4.4]{BBSV13}  and used in the formulation of \cite[Lemma~2.3]{KL17}, we will formulate and prove new stronger bounds of these operators in this subsection.
%
% For $\Re s > \sigma >0$, we have
%
%\begin{align*}
%	\norm{V(s)}_{\mXG \leftarrow \mXG}&\le C_{\sigma}\abs{s}^2,\\
%	\norm{K(s)}_{\mXG \leftarrow \mXG}&\le C_{\sigma}\abs{s}^2.
%\end{align*}
The \emph{Calder\'{o}n operator} is a block operator consisting of these boundary operators and has, with a different scaling with respect to the magnetic permeability, been introduced in the dielectric setting (i.e. real-valued and positive $\varepsilon$ and $\mu$) by \cite{KL17} (note the sign correction from \cite{NKL2020}). In the present setting we obtain the following Calder\'{o}n operator, which reads 
\begin{equation} \label{def-B}
\Cald\epsmu(s)=\begin{pmatrix}
-\frac{\sqrt{\mu(s)}}{\sqrt{\varepsilon(s)}}	\V\epsmu(s) & \K\epsmu(s) \\
-\K\epsmu(s) & -\frac{\sqrt{\varepsilon(s)}}{{\sqrt{\mu(s)}}}	\V\epsmu(s)
\end{pmatrix},
%= \avg{\gaT}\circ
%\begin{pmatrix}
%-\frac{\sqrt{\mu(s)}}{\sqrt{\varepsilon(s)}}	\mathbfcal{S}\epsmu(s) & \mathbfcal{D}\epsmu(s) \\
%-\mathbfcal{D}\epsmu(s) & -\frac{\sqrt{\varepsilon(s)}}{\sqrt{\mu(s)}}	\mathbfcal{S}\epsmu(s)
%\end{pmatrix},
%:\mXG\times\mXG \rightarrow \mXG\times\mXG .
\end{equation}
where the form of the block operator on the right originates in the representation formula \eqref{eq:time-harmonic-kirchhoff-E3}--\eqref{eq:time-harmonic-kirchhoff-H3}. Consider outgoing solutions of the time-harmonic Maxwell's equations $\widehat \E,\widehat \B $, thus characterized by the representation formulas. The composition of the tangential averages with the representation formulas reveals the jump relations of 
the Calder\'on operator (see \eqref{eq:transmis-1-thMW}--\eqref{eq:transmis-4-thMW}):
%(details can be found in \cite{KL17}) the property
\begin{align}\label{eq:calderon-jump}
\Cald\epsmu(s)\begin{pmatrix}\jmp{\gaT }\widehat \B\\ -\jmp{\gaT} \widehat \E \end{pmatrix}=
\begin{pmatrix}
\avg{\gaT}\widehat{\E} \\ \avg{\gaT}\widehat{\B}
\end{pmatrix} .
\end{align}
%\begin{proof} 
%	Let $(\widetilde\bvarphi,\widetilde\bpsi)=(\sqmu\widehat\bvarphi,\sqepsInv\widehat\bpsi)$ and define $(\widetilde \E,\mu\widetilde \H)$ through the representation formulas \eqref{eq:time-harmonic-kirchhoff-E3}--\eqref{eq:time-harmonic-kirchhoff-H3}. The jump conditions of the time-harmonic transmission problem \eqref{eq:transmis-3-thMW}--\eqref{eq:transmis-4-thMW} then reads
%\begin{align*}
%	\,\jmp{\gaT}\mu\widetilde{\B}&= \widetilde\bvarphi = \sqmu \bwvarphi\,, \\
%	-\jmp{\gaT} \widetilde \E&=\widetilde\bpsi = \sqepsInv \widehat\bpsi\,. 
%\end{align*}
%	Rearranging yields $(\bwvarphi,\bwpsi) = \left(\jmp{\gaT}\left(\sqmu\widetilde{\B}\right),-\jmp{\gaT}\left(\sqeps \widetilde \E\right)\right) $
%\end{proof}
%
The application of this operator is thus equivalent to transform jumps of the transmission problem to averages, which directly implies bounds from above through Lemma~\ref{lem:transmission}.

As a direct consequence, we obtain the following bound, equivalent to \cite[Lemma~3.4]{NKL22} in the dielectric case.
Earlier, slightly different bounds can be found in the dielectric case in \cite[Theorem 4.4]{BBSV13} and~\cite[Lemma~2.3]{KL17}, which are of the order $O(|s|^2)$ .
\begin{lemma}
	\label{lemma:B-boundedness}
	For $s$ with positive real part, the Calder\'on operator is a linear operator family on the trace space $\Cald\epsmu(s):\mXG^2 \to \mXG^2$ and satisfies the bound
	\begin{equation}
	\label{calderon-strong-bound}
	\norm{\Cald\epsmu(s)}_{\mXG^2\leftarrow\mXG^2}\le  \Ctrace	\max\left(\frac{|\varepsilon(s) s|^2+1}{\Re \varepsilon(s) s}\,,\,\frac{|\mu(s) s|^2+1}{\Re \mu(s) s}\right) ,
	\end{equation}
	where the constant is the norm of the tangential average $\Ctrace=\| \{\gaT\} \|_{\mXG \leftarrow \H(\curl,\mathbb{R}^3\setminus\Gamma)}$.
	The identical bound holds for the components of the Calderón operator \eqref{def-B} and for the 
	 electromagnetic single and double layer boundary operators $	\V\epsmu(s)$ and $\K\epsmu(s)$.
	 %, i.e. on the expressions
%	\textcolor{blue}{
%	\begin{equation*}
%	\norm{{\scriptstyle\frac{\sqrt{\mu(s)}}{\sqrt{\eps(s)}}}\V\epsmu(s)}_{\mXG\leftarrow \mXG}+\norm{\K\epsmu(s)}_{\mXG\leftarrow \mXG}
%	\end{equation*}
%and 
%\begin{equation*}
%\norm{{\scriptstyle\frac{\sqrt{\eps(s)}}{\sqrt{\mu(s)}}}\V\epsmu(s)}_{\mXG\leftarrow \mXG}+\norm{\K\epsmu(s)}_{\mXG\leftarrow \mXG}.
%\end{equation*}
%}
	
	%	\begin{equation}
	%	\norm{\V(s)}_{\mXG\leftarrow \mXG}+\norm{K(s)}_{\mXG\leftarrow \mXG}\le \Ctrace^2 \dfrac{\abs{s}^2+1}{\Re s}
	%	\end{equation}
	%	hold.
	%	The constant $\Ctrace$ is, up to a multiplication with the maximum physical constants $\varepsilon$ and $\mu$, the norm of the trace operator $\gaT:H(\curl,\mathbb{R}^3\setminus\Gamma)\rightarrow \mXG$.
\end{lemma}

%\begin{proof}
%	Let $(\wvarphi,\wpsi)\in \mXG\times\mXG$ be arbitrary and $\widehat{\E},\widehat{\H} \in \H(\curl,\mathbb{R}^3\setminus \Gamma)$ be the solutions to the associated transmission problem of Lemma~\ref{lem:transmission}.
%	Applying consecutively the jump condition of the Calder\'on operator \eqref{eq:calderon-jump}, the trace theorem \cite[Theorem~4.1]{BCS02} (with constant $\Ctrace$), and finally the bound \eqref{energy-estimate} yields
%	\begin{align*}
%	\norm{\Cald(s)
%		\begin{pmatrix}
%		\wvarphi\\ \wpsi
%		\end{pmatrix}
%	}^2_{\mXG\times\mXG} 	= &\
%	\norm{
%		\begin{pmatrix}
%		\avg{\gaT\widehat{\E}} \\ \mu\avg{\gaT\widehat{\H}}
%		\end{pmatrix}}_{\mXG\times\mXG}^2
%	\\
%	\leq &\ {\Ctrace^2}\left(\big\| \widehat{\E} \big\|_{\H(\curl,\mathbb{R}^3\setminus \Gamma)}^2 +
%	\big\| \mu \widehat{\H} \big\|_{\H(\curl,\mathbb{R}^3\setminus \Gamma)}^2\right)
%%	\\
%%	\leq &\ \Ctrace^2\mu \dfrac{\abs{s}^2+1}{\Re s}\abs{\left[
%%		\begin{pmatrix} \wvarphi\\ \wpsi \end{pmatrix},
%%		\Cald(s)
%%		\begin{pmatrix} \wvarphi\\ \wpsi \end{pmatrix}\right]_\Gamma}
%	\\
%	\leq &\ \Ctrace^2\, \dfrac{\abs{s}^2+1}{\Re s}\norm{\Cald(s)
%	}_{\mXG^2\leftarrow \mXG^2}
%	\norm{
%		\begin{pmatrix}
%		\wvarphi\\ \wpsi
%		\end{pmatrix}
%	}^2_{\mXG\times\mXG}.
%	\end{align*}
%	Now, taking the supremum over all $(\wvarphi,\wpsi)\in\mXG\times\mXG$ of unit norm and then dividing by the operator norm on both sides yields the stated result.
%\qed
%\end{proof}
%

The skew-hermitian pairing $[\cdot,\cdot]_\Gamma$ is notationally extended from $\mXG\times \mXG$ to $\mXG^2\times \mXG^2$ in the natural way:
\begin{equation}%\label{eq:pairing-extended}
\left[ \begin{pmatrix} \bvar \\ \bpsi \end{pmatrix} , \begin{pmatrix} \bupsilon \\ \bxi \end{pmatrix} \right]_\Gamma =
[ \bvar,\bupsilon]_\Gamma + [ \bpsi,\bxi]_\Gamma.
\end{equation}
As was shown in \cite[Lemma~3.1]{KL17} in the dielectric case with positive and real-valued $\varepsilon$ and $\mu$, the Calder\'on operator $\Cald(s)$ is positive with respect to this extended skew-symmetric pairing $[\cdot,\cdot]_\Gamma$. The following lemma transfers this key property to the present setting of analytic $\varepsilon (s)$ and $\mu(s)$.

%, and we include the short proof that  reframes the arguments of the proof in~\cite{KL17} within the present setting.

\begin{lemma}
	\label{lem:B-coercivity}
	The Calder\'on operator is of positive type: for $\Re s>0$, 
	\begin{equation}
	\label{eq:coercivity Calderon}
	\textnormal{Re} \left[
	\begin{pmatrix} \bvar\\ \bpsi \end{pmatrix}
	,\Cald\epsmu(s)
	\begin{pmatrix} \bvar\\ \bpsi \end{pmatrix}
	\right]_\Gamma
	\ge c_\Gamma^{-2} \min\left(\frac{\Re \varepsilon(s) s}{|\varepsilon(s) s|^2+1}\,,\,\frac{\Re \mu(s) s}{|\mu(s) s|^2+1}\right)
	\,
	 \left(\bigl\|{\bvar}\bigr\|^2_{\mXG}+\bigl\|\bpsi\bigr\|^2_{\mXG}\right)
	%		\geq C_\Gamma \min(\sigma,\sigma^3\varepsilon\mu) \textnormal{Re } s\left((\varepsilon\mu)^{-1}\norm{s^{-1}\bvar}^2_{\mVG}+\norm{s^{-1}\bpsi}^2_{\mXG}\right),
	%	\geq  c_\sigma \, \Re s \left(\bigl\|{s^{-1}\bvar}\bigr\|^2_{\mXG}+\bigl\|s^{-1}\bpsi\bigr\|^2_{\mXG}\right),
	\end{equation}
	for all $(\bvar,\bpsi) \in \mXG^2 $. The constant is the norm of the jump operator associated to the tangential trace, i.e. $\ctrace=\| [\gaT] \|_{\mXG \leftarrow \H(\curl,\mathbb{R}^3\setminus\Gamma)}$.
\end{lemma}

\begin{proof}
	Consider $(\bwvarphi,\bwpsi)\in \mXG^2$ and let the time-harmonic fields ${\widehat{\E},\widehat{\B} \in \H(\curl,\mathbb{R}^3\setminus \Gamma)}$ be given through the representation formula, therefore solving the associated transmission problem of Lemma~\ref{lem:transmission}. The result is then given by the following chain of inequalities, taken from the proof of \cite[Lemma 3.5]{NKL22} 
	%	consecutively by
	%	\eqref{eq:transmis-3}--\eqref{eq:transmis-4}, by the bound $\ctrace$ of the jump operator $[\gaT]$,
	%	and by \eqref{ReI-1}--\eqref{ReI-2},
	\begin{align*}
	&\norm{\begin{pmatrix}
		\bwvarphi\\ \bwpsi
		\end{pmatrix}
	}^2_{\mXG\times\mXG}
	=
	\norm{\begin{pmatrix}\jmp{\gaT }\widehat \B
		\\ -\jmp{\gaT} \widehat \E\end{pmatrix}
	}^2_{\mXG\times\mXG}
	& \quad\text{by~\eqref{eq:transmis-3-thMW}--\eqref{eq:transmis-4-thMW}}
	\\
	& \le \ctrace^2 \left(  \big\| \widehat{\B} \big\|_{\H(\curl,\mathbb{R}^3\setminus \Gamma)}^2
	+ \big\| \widehat{\E} \big\|_{\H(\curl,\mathbb{R}^3\setminus \Gamma)}^2
	\right)
	&\quad\text{by def.~of $c_\Ga$}
	\\
	&\le \ctrace^2 \max\left(\frac{|\varepsilon s|^2+1}{\Re \varepsilon s}\,,\,\frac{|\mu s|^2+1}{\Re \mu s}\right)	\,
	\Re \left[
	%\begin{pmatrix} \wvarphi\\ \wpsi \end{pmatrix},
	\begin{pmatrix}\jmp{\gaT }\widehat \B \\ -\jmp{\gaT} \widehat \E \end{pmatrix},
	\begin{pmatrix}\avg{\gaT}\widehat{\E} \\ \avg{\gaT}\widehat{\B} \end{pmatrix}
	\right]_\Gamma
	&\quad \text{by \eqref{ReI-1}--\eqref{ReI-2}}
	\\
	&= \ctrace^2 \max\left(\frac{|\varepsilon s|^2+1}{\Re \varepsilon s}\,,\,\frac{|\mu s|^2+1}{\Re \mu s}\right)	
	\,\Re \left[
	\begin{pmatrix}\jmp{\gaT }\widehat \B\\ -\jmp{\gaT} \widehat \E \end{pmatrix},
	%\begin{pmatrix} \wvarphi\\ \wpsi \end{pmatrix},
	%\Cald(s)
	%\begin{pmatrix} \wvarphi\\ \wpsi \end{pmatrix}\right]_\Gamma\,,
	\Cald\epsmu(s)\!\begin{pmatrix}\jmp{\gaT }\widehat \B\\ -\jmp{\gaT} \widehat \E \end{pmatrix}\right]_\Gamma
	&\quad\text{by \eqref{eq:calderon-jump}}
	\\
	&= \ctrace^2 \max\left(\frac{|\varepsilon s|^2+1}{\Re \varepsilon s}\,,\,\frac{|\mu s|^2+1}{\Re \mu s}\right)		\,
	\Re \left[
	\begin{pmatrix} \bwvarphi\\ \bwpsi \end{pmatrix},
	%\Cald(s)
	%\begin{pmatrix} \wvarphi\\ \wpsi \end{pmatrix}\right]_\Gamma\,,
	\Cald\epsmu(s)\! \begin{pmatrix} \bwvarphi\\ \bwpsi \end{pmatrix}\right]_\Gamma
	&\quad\text{by~\eqref{eq:transmis-3-thMW}--\eqref{eq:transmis-4-thMW}}.
	\end{align*}
	
	%where the last equality follows from \eqref{eq:calderon-jump} on inserting \eqref{eq:transmis-3}--\eqref{eq:transmis-4}.
\end{proof}

\section{The time-harmonic scattering problem}
\label{sec:th-scattering}

The time-harmonic problem formulation reads
\begin{alignat}{2}
\begin{aligned}
\label{TH-MW1-transmission}
 \varepsilon^{\pm}(s)  s\widehat{\E}^{\pm}-\curl \widehat{\B}^{\pm} &= 0
%	\widehat{\J}
\\
 \mu^{\pm}(s)  s\widehat{\B}^{\pm} + \curl \widehat{\E}^{\pm} &=0 
\end{aligned}
\qquad \text{in} \ \Omega^{\pm} ,
\end{alignat}
completed by the transmission conditions, which enforce the continuity of the time-harmonic electromagnetic fields $\widehat \E$ and $\widehat \B$:
\begin{align}\label{eq:transmission-conditions}
\begin{aligned}
\gamma_T\widehat \E^+ + \gamma_T\widehat \E^+\inc &= \gamma_T\widehat \E^-\\
\gamma_T\widehat \H^+ + \gamma_T\widehat \H^+\inc &= \gamma_T\widehat \H^-
\end{aligned}\qquad \text{on} \ \Gamma.
\end{align}
\subsection{The time-harmonic boundary integral equation}
In this subsection we derive the time-harmonic boundary integral equation, which determines the boundary densities to be inserted into the representation formulas for the electromagnetic fields. Assuming that we are given solutions to the time-harmonic Maxwell's equations in the exterior or interior domain $\Omega^\pm$, we obtain solutions on $\R^3\setminus \Gamma$ by extension to zero on $\Omega^\mp$. Then, jumps and averages reduce to outer or inner traces, respectively. We start by collecting the (supposed) solutions of the boundary integral equations in the vectors

\begin{align*}
	\bwphip = \begin{pmatrix}
	\bwvarphi^+ \\ \bwpsi^+
	\end{pmatrix} =\begin{pmatrix}
	\gamma_T^+\widehat{\H}^+ \\ -\gamma_T^+\widehat{\E}^+
	\end{pmatrix},
	\quad \quad
	\bwphim = \begin{pmatrix}
	\bwvarphi^- \\ \bwpsi^-
	\end{pmatrix} =\begin{pmatrix}
	-\gamma_T^-\widehat{\H}^- \\ \gamma_T^-\widehat{\E}^-
	\end{pmatrix},
\end{align*} 
%and furthermore
%\begin{align*}
%\bwphim = \begin{pmatrix}
%\bwvarphi^- \\ \bwpsi^-
%\end{pmatrix} =\begin{pmatrix}
%\gamma_T^-\widehat{\H}^- \\ -\gamma_T^-\widehat{\E}^-
%\end{pmatrix}.
%\end{align*} 
and denote the block operator $\bJ$ and the trace of the incoming wave $\bwginc$ by
\begin{align*}
\bJ = \frac{1}{2}\begin{pmatrix} & -\bId \\ \bId & \end{pmatrix}, \qquad \bwginc = \frac{1}{2}\begin{pmatrix} \gamma^+_T\widehat \E^+\inc \\ \gamma^+_T\widehat \H^+\inc \end{pmatrix}.
\end{align*}
In order to derive the boundary integral equation, we first use \eqref{eq:calderon-jump}, followed by the transmission conditions \eqref{eq:transmission-conditions}. This yields
\begin{align*}
 \Cald_{\varepsilon^+,\mu^+}(s)	\bwphip &= \frac12
 \begin{pmatrix}
 \gamma_T^+\widehat{\E}^+ \\ \gamma_T^+\widehat{\H}^+
 \end{pmatrix}
 = \frac12 
  \begin{pmatrix}
 \gamma_T^-\widehat{\E}^- \\ \gamma_T^-\widehat{\H}^-
 \end{pmatrix}
 -\bwginc
= -\bJ \bwphim -\bwginc,
\end{align*}
and
\begin{align*}
\Cald_{\varepsilon^-,\mu^-}(s)	\bwphim 
=\frac12 
\begin{pmatrix}
	\gamma_T^-\widehat{\E}^- \\ \gamma_T^-\widehat{\H}^-
\end{pmatrix}
 = \frac12 
 \begin{pmatrix}
	\gamma_T^+\widehat{\E}^+ \\ \gamma_T^+\widehat{\H}^+
\end{pmatrix}
-\bwginc
= \bJ \bwphip +\bwginc.
\end{align*}
Introducing the family of operators $\bA(s) : \mXG^4 \to \mXG^4$ defined as 
\begin{align}\label{eq:def-A}
\bA(s) 
\coloneqq
\begin{pmatrix}
\Cald_{\varepsilon^+,\mu^+}(s)  & \bJ \\
-\bJ  &
\Cald_{\varepsilon^-,\mu^-}(s) 
\end{pmatrix},
\end{align}
we arrive at the time-harmonic boundary integral equation
\begin{align}\label{eq:bie-A}
\bA(s) \begin{pmatrix}
\bwphip \\ \bwphim
\end{pmatrix} 
%\coloneqq
%\begin{pmatrix}
% \Cald_{\varepsilon^+,\mu^+}(s)  & \bJ \\
%-\bJ  &
%   \Cald_{\varepsilon^-,\mu^-}(s) 
%\end{pmatrix}
%\begin{pmatrix}
%\bwphip \\ \bwphim
%\end{pmatrix}
=
\begin{pmatrix}
-\bwginc \\ \bwginc
\end{pmatrix}.
%\quad\text{where}\quad
%\bA(s) 
%\coloneqq
%\begin{pmatrix}
%\Cald_{\varepsilon^+,\mu^+}(s)  & \bJ \\
%-\bJ  &
%\Cald_{\varepsilon^-,\mu^-}(s) 
%\end{pmatrix}
\end{align}
This boundary integral equation will be considered in its weak formulation:
For $\Re s > 0$ and given $\bwginc \in \mXG^2$, find $(\bphip, \bphim) \in \mXG^4$ such that, for all $(\bupsilon, \bxi) \in \mXG^4$
\begin{equation}\label{eq:bie-A-weak}
\left[
\begin{pmatrix}
\bupsilon \\ \bxi
\end{pmatrix}
,\bA(s)
\begin{pmatrix}
\bphip \\ \bphim
\end{pmatrix}
\right]_\Gamma = \left[
\begin{pmatrix}
\bupsilon \\ \bxi
\end{pmatrix}
,
\begin{pmatrix}
-\bwginc \\ \bwginc
\end{pmatrix}
\right]_\Gamma.
\end{equation}
%where $[\cdot, \cdot]_\Gamma$ denotes the anti-duality between $\mXG^4$ and itself and the family of operators $\bA(s) : \mXG^4 \to \mXG^4$ is defined as 
%\begin{align}\label{eq:def-A}
%\bA(s) 
%\coloneqq
%\begin{pmatrix}
%\Cald_{\varepsilon^+,\mu^+}(s)  & \bJ \\
%-\bJ  &
%\Cald_{\varepsilon^-,\mu^-}(s) 
%\end{pmatrix}.
%\end{align}
%We write this more compactly as
%\begin{align}\label{eq:bie-A}
%\bA(s) \begin{pmatrix}
%\bwphip \\ \bwphim
%\end{pmatrix} 
%%\coloneqq
%%\begin{pmatrix}
%% \Cald_{\varepsilon^+,\mu^+}(s)  & \bJ \\
%%-\bJ  &
%%   \Cald_{\varepsilon^-,\mu^-}(s) 
%%\end{pmatrix}
%%\begin{pmatrix}
%%\bwphip \\ \bwphim
%%\end{pmatrix}
%=
%\begin{pmatrix}
%-\bwginc \\ \bwginc
%\end{pmatrix}.
%%\quad\text{where}\quad
%%\bA(s) 
%%\coloneqq
%%\begin{pmatrix}
%%\Cald_{\varepsilon^+,\mu^+}(s)  & \bJ \\
%%-\bJ  &
%%\Cald_{\varepsilon^-,\mu^-}(s) 
%%\end{pmatrix}
%\end{align}
%
%%%%%
%%
%%\begin{align}%\label{eq:bie-A}
%%\bA(s) \begin{pmatrix}
%%	\bwphip \\ \bwphim
%%\end{pmatrix} 
%%=
%%\begin{pmatrix}
%%-\bwginc \\ \bwginc
%%\end{pmatrix}
%%\quad\text{where}\quad
%%\bA(s) 
%%\coloneqq
%%\begin{pmatrix}
%%	\Cald_{\varepsilon^+,\mu^+}(s)  & \bJ \\
%%	-\bJ  &
%%	\Cald_{\varepsilon^-,\mu^-}(s) 
%%\end{pmatrix}
%%\end{align}
%%
Crucially, the bilinear form on the left-hand side is coercive, as will be shown next.

\subsection{Well-posedness of the boundary integral equation}\label{sec:wp-bie}
This section is dedicated to the well-posedness of the time-harmonic boundary integral equation, which is shown by employing the Lax-Milgram Lemma. To simplify the expressions in this section, we use the abbreviation
\begin{equation}\label{eq:max-term-in-bounds}
m_{\epsmu}(s) \coloneqq \max\left(\frac{|\varepsilon^+(s) s|^2+1}{\Re \varepsilon^+(s) s}\,,\,\frac{|\mu^+(s) s|^2+1}{\Re \mu^+(s) s},\frac{|\varepsilon^-(s) s|^2+1}{\Re \varepsilon^-(s) s}\,,\,\frac{|\mu^-(s) s|^2+1}{\Re \mu^-(s) s}\right).
\end{equation}
Under the strong passivity condition \eqref{strongly-passive}  there is a convenient upper bound for $m_{\epsmu}(s) $:
\eqref{strongly-passive}--\eqref{eq:bound-eps-mu} imply that for every $\sigma>0$ there exists $C_\sigma<\infty$ such that
\begin{equation} \label{mepsmu-bound}
m_{\epsmu}(s) \le C_\sigma \,\frac{|s|^2}{\Re s} \quad\text{ for }\ \Re s \ge \sigma.
\end{equation}
We start by giving a bound for the boundary integral operator.

\begin{lemma}
	\label{lemma:A-boundedness}
	The analytic operator family $\bA(s)\colon \mXG^4\leftarrow\mXG^4$ satisfies, for $\Re\,s>0$, the bound
	\begin{align*}
	\label{A-strong-bound}
	\norm{\bA(s)}_{\mXG^4\leftarrow\mXG^4}\le \Ctrace	
	\, m_{\epsmu}(s)
%	\max\left(\frac{|\varepsilon^+ s|^2+1}{\Re \varepsilon^+ s}\,,\,\frac{|\mu^+ s|^2+1}{\Re \mu^+ s},\frac{|\varepsilon^- s|^2+1}{\Re \varepsilon^- s}\,,\,\frac{|\mu^- s|^2+1}{\Re \mu^- s}\right)
	+  \frac12 ,
	\end{align*}
	where $\Ctrace=\| \{\gaT\} \|_{\mXG \leftarrow \H(\curl,\mathbb{R}^3\setminus\Gamma)}$  is the norm of the tangential average.

	%	\begin{equation}
	%	\norm{\V(s)}_{\mXG\leftarrow \mXG}+\norm{K(s)}_{\mXG\leftarrow \mXG}\le \Ctrace^2 \dfrac{\abs{s}^2+1}{\Re s}
	%	\end{equation}
	%	hold.
	%	The constant $\Ctrace$ is, up to a multiplication with the maximum physical constants $\varepsilon$ and $\mu$, the norm of the trace operator $\gaT:H(\curl,\mathbb{R}^3\setminus\Gamma)\rightarrow \mXG$.
\end{lemma}

Moreover, we have the following coercivity result for the integral operator corresponding to the boundary integral equation.
\begin{lemma}
	\label{lem:A-coercivity}
	The operator family $\bA(s)$ satisfies the following coercivity property:
	for  $\Re s >0$ we have the bound
	\begin{align*}
\Re \left[
\begin{pmatrix}
\bphip \\ \bphim
\end{pmatrix}
,\bA(s)
\begin{pmatrix}
\bphip \\ \bphim
\end{pmatrix}
\right]_\Gamma
\ge&\
c_\Gamma^{-2}
m_{\epsmu}(s)^{-1}\left(\bigl\|{\bphip}\bigr\|^2_{\mXG^2} + \bigl\|\bphim\bigr\|^2_{\mXG^2}\right),
\end{align*}
%
%	\begin{align*}
%	\Re \left[
%\begin{pmatrix}
%\bphip \\ \bphim
%\end{pmatrix}
%	,\bA(s)
%\begin{pmatrix}
%\bphip \\ \bphim
%\end{pmatrix}
%	\right]_\Gamma
%	\ge&
%	  c_\Gamma^{-2} \min\left(\frac{\Re \varepsilon^+(s) s}{|\varepsilon^+(s) s|^2+1}\,,\,\frac{\Re \mu^+(s) s}{|\mu^+(s) s|^2+1}\right)
%   \bigl\|{\bphip}\bigr\|^2_{\mXG^2}
%    \\ +&
%	  c_\Gamma^{-2} \min\left(\frac{\Re \varepsilon^-(s) s}{|\varepsilon^-(s) s|^2+1}\,,\,\frac{\Re \mu^-(s) s}{|\mu^-(s) s|^2+1}\right)
%    \,
%    \bigl\|\bphim\bigr\|^2_{\mXG^2}.
%	%	, \qquad \text{for all } \ (\bvar,\bpsi) \in \mVG\times \mXG,
%	\end{align*}
	for all $ (\bphip , \bphim) \in \mXG^2\times \mXG^2$, where $\ctrace=\| \jmp{\gaT} \|_{\mXG \leftarrow \H(\curl,\mathbb{R}^3\setminus\Gamma)}$  is the norm
%	The constant $c_\Gamma = \norm{\jmp{\gamma_T}}$ denotes the operator norm 
of the tangential jump operator. 
	%	The constant $c_\sigma$ only depends on the surface $\Gamma $ and on $\sigma$.
\end{lemma}
\begin{proof}
 We split the operator in the pairing
\begin{align*}
	\Re \left[
	\begin{pmatrix}
	\bphip \\ \bphim
	\end{pmatrix}
	,\bA(s)
	\begin{pmatrix}
	\bphip \\ \bphim
	\end{pmatrix}
	\right]_\Gamma
	&=
	\Re \left[
	\begin{pmatrix}
	\bphip \\ \bphim
	\end{pmatrix}
	,\begin{pmatrix}
	\Cald_{\varepsilon^+,\mu^+}(s)  & \\
	 &
	\Cald_{\varepsilon^-,\mu^-}(s) 
	\end{pmatrix}
	\begin{pmatrix}
	\bphip \\ \bphim
	\end{pmatrix}
	\right]_\Gamma	
\\ & +
	\Re \left[
	\begin{pmatrix}
	\bphip \\ \bphim
	\end{pmatrix}
	,\begin{pmatrix}
  & \bJ \\
	-\bJ  &
	\end{pmatrix}
	\begin{pmatrix}
	\bphip \\ \bphim
	\end{pmatrix}
	\right]_\Gamma,
\end{align*}
where the first summand is bounded from below by the coercivity of the Calderón operators given in Lemma~\ref{lem:B-coercivity}. The second summand vanishes due to symmetry of $\bJ$, which we verify next. We have
\begin{align*}%\label{eq:symmetry-J}
	\begin{split}
		2\,\Re \left[\widetilde \bphi\,,\, \bJ \bphi\right]_\Gamma	
		%=\Re \left[\begin{pmatrix} & -\bId \\ \bId & \end{pmatrix}\widetilde \bphi\,,\,  \bphi\right]_\Gamma
%		&=
%		\Re \left[\widetilde \bphi_2\,,\, \bphi_1\right]_\Gamma
%		+\Re\left[\widetilde \bphi_1\,,\, -\bphi_2\right]_\Gamma
%		\\ 
		&=
		\Re \left[ \bphi_1\,,\,-\widetilde \bphi_2\right]_\Gamma
		+\Re\left[\bphi_2\,,\, \widetilde \bphi_1\right]_\Gamma =
		2\Re \left[\bphi\,,\, \bJ \widetilde \bphi\right]_\Gamma,
	\end{split}
\end{align*}
%
%\begin{align*}%\label{eq:symmetry-J}
%\begin{split}
%\Re \left[\widetilde \bphi\,,\, \begin{pmatrix} & -\bId \\ \bId & \end{pmatrix} \bphi\right]_\Gamma	
%%=\Re \left[\begin{pmatrix} & -\bId \\ \bId & \end{pmatrix}\widetilde \bphi\,,\,  \bphi\right]_\Gamma
%&=
%\Re \left[\widetilde \bphi_2\,,\, \bphi_1\right]_\Gamma
%+\Re\left[\widetilde \bphi_1\,,\, -\bphi_2\right]_\Gamma
%\\ &=
%\Re \left[ \bphi_1\,,\,-\widetilde \bphi_2\right]_\Gamma
%+\Re\left[\bphi_2\,,\, \widetilde \bphi_1\right]_\Gamma \\
%&=
%\Re \left[\bphi\,,\, \begin{pmatrix} & -\bId \\ \bId & \end{pmatrix} \widetilde \bphi\right]_\Gamma.
%\end{split}
%\end{align*}
%
and therefore,
\begin{align*}
	\Re \left[
	\begin{pmatrix}
	\bphip \\ \bphim
	\end{pmatrix}
	,\begin{pmatrix}
	& \bJ \\
	-\bJ  &
	\end{pmatrix}
	\begin{pmatrix}
	\bphip \\ \bphim
	\end{pmatrix}
	\right]_\Gamma 
	&=
	\Re \left[  \bphip,\bJ\bphim\right]_\Gamma -	\Re \left[  \bphim,\bJ\bphip\right]_\Gamma =0,
\end{align*}
such that the claim follows.
\end{proof}

In view of this coercivity, we obtain the following well-posedness result.
\begin{proposition} [Well-posedness of the time-harmonic boundary integral equation]\label{prop:bie}
Consider the boundary integral equation~\eqref{eq:bie-A} for $\Re s > 0$. The boundary integral equation has a unique solution 
$$
\begin{pmatrix}\bwphip\\ \bwphim\end{pmatrix}= \bA(s)^{-1} 	\begin{pmatrix}-\bwginc\\ \bwginc\end{pmatrix} \in\mXG^4,
$$ 
which satisfies
\begin{equation} \label{phi-psi-bound-th}
\norm{
	\begin{pmatrix}\bwphip\\ \bwphim\end{pmatrix}
}_{\mXG^4}
\le %\dfrac{C_\sigma}{\mathrm{Re } \ s}
\ctrace^2 \,
m_{\epsmu}(s)\,\sqrt2\,
%\max\left(\frac{|\varepsilon^+(s) s|^2+1}{\Re \varepsilon^+(s) s}\,,\,\frac{|\mu^+(s) s|^2+1}{\Re \mu^+(s) s},\frac{|\varepsilon^-(s) s|^2+1}{\Re \varepsilon^-(s) s}\,,\,\frac{|\mu^-(s) s|^2+1}{\Re \mu^-(s) s}\right)\,
\norm{\bwginc}_{\mXG^2}.
\end{equation}
The constant  $\ctrace$ is again the norm of the tangential jump operator, and $m_{\epsmu}(s)$ is defined in \eqref{eq:max-term-in-bounds}.
\end{proposition}
\begin{proof}
	The statement follows directly from the Lax--Milgram lemma with the coercivity of Lemma~\ref{lem:A-coercivity}.
\end{proof}
%

%Aus 
Using the above properties, we prove the following result, where the domain $\Om$ stands for either $\Om^+$ or $\Om^-$.

\begin{proposition}[Well-posedness of the time-harmonic scattering problem]\label{prop:wp-time-harm-scattering}
%	(Well-posedness of the time-harmonic scattering problem)
	For $\Re s >0$ there exists a unique solution$(\widehat{\E}, \widehat{\B}) \in \H(\curl, \Omega)\times \H(\curl, \Omega)$ to the time-harmonic transmission problem \eqref{TH-MW1-transmission}--\eqref{eq:transmission-conditions} given by the representation formulas \eqref{eq:time-harmonic-kirchhoff-E3}--\eqref{eq:time-harmonic-kirchhoff-H3}. The tangential traces are given by the unique solution of the boundary integral equation \eqref{eq:bie-A} via
	\begin{align*}
	\bwphip = \begin{pmatrix}
	\gamma_T^+\widehat{\H}^+ \\ -\gamma_T^+\widehat{\E}^+
	\end{pmatrix},
	\quad \quad
	\bwphim =\begin{pmatrix}
	-\gamma_T^-\widehat{\H}^- \\ \gamma_T^-\widehat{\E}^-
	\end{pmatrix} .
	\end{align*} 
The scattered electromagnetic fields are bounded by 
%\begin{equation}\label{eq:electromag-fields-bound-improved}
%\norm{\widehat{\E}}_{\H(\curl, \Om)} + \norm{\widehat{\B}}_{\H(\curl, \Om)} \le
%\end{equation}
	\begin{align}\label{eq:electromag-fields-bound-improved}
\left\|\begin{pmatrix} \widehat{\E} \\ \widehat{\B} \end{pmatrix} \right\|_{\H(\curl,\R^3\setminus \Gamma)^2}
\le  \Ctrace	
\left(m_{\epsmu}(s)\right)^{3/2}
\left\|\bwginc  \right\|_{\mXG^2}.
\end{align}	

\end{proposition}
\begin{proof}
	Let $(\bwvarphi^+, \bwpsi^+, \bwvarphi^-, \bwpsi^-)$ be the solution of the time-harmonic boundary integral equation. We insert the boundary densities into the representation formulas and obtain electromagnetic fields $(\wE^{+},\wE^{-},\wH^{+},\wH^{-})$, each defined on $\mathbb R^3\setminus \Gamma$, such that
	\begin{equation*}
	\bwphip = 	\begin{pmatrix}
	\bwvarphi^+\\ \bwpsi^+ 
	\end{pmatrix}
	=
	\begin{pmatrix}
	\jmp{\gamma_T}\widehat{\H}^+\\ -\jmp{\gamma_T}\widehat{\E}^+
	\end{pmatrix}
	\quad\text{and}\quad 
	\bwphim = 
	\begin{pmatrix}
	\bwvarphi^-\\\bwpsi^-
	\end{pmatrix}
	=
	\begin{pmatrix} \jmp{\gamma_T}\widehat{\H}^-\\  -\jmp{\gamma_T}\widehat{\E}^-
	\end{pmatrix}.
	\end{equation*}
	The first two components of the left-hand side of boundary integral equation read
	\begin{align}\label{eq:component-1}
	-\bwginc &= 	\Cald_{\varepsilon^+,\mu^+}(s)	\bwphip +\bJ \bwphim
	= 
	\begin{pmatrix}
	\avg{\gamma_T}\widehat{\E}^+ \\ \avg{\gamma_T}\widehat{\H}^+
	\end{pmatrix}
	+ \frac12 
	\begin{pmatrix}
	\jmp{\gamma_T}\widehat{\E}^- \\ \jmp{\gamma_T}\widehat{\H}^-
	\end{pmatrix}
%	\\&= \frac12 
%	\begin{pmatrix}
%	\left(\gamma_T^+\widehat{\E}^+-		\gamma_T^-\widehat{\E}^-\right)	+\gamma_T^-\widehat{\E}^+
%	+ \gamma_T^+\widehat{\E}^-
%	\\
%	\left(\gamma_T^+\widehat{\H}^+-		\gamma_T^-\widehat{\H}^-\right)	+\gamma_T^-\widehat{\H}^+
%	+ \gamma_T^+\widehat{\H}^-
%	\end{pmatrix},
	\end{align}
	and
	\begin{align}\label{eq:component-2}
	\bwginc &= -\bJ \bwphip	+\Cald_{\varepsilon^-,\mu^-}(s)	\bwphim 
	=\frac12 \begin{pmatrix}
	-\jmp{\gamma_T}\widehat{\E}^+\\ -\jmp{\gamma_T}\widehat{\H}^+
	\end{pmatrix}
	+
	\begin{pmatrix}
	\avg{\gamma_T}\widehat{\E}^- \\ \avg{\gamma_T}\widehat{\H}^-
	\end{pmatrix}	.
%	\\&	= \frac12 
%	\begin{pmatrix}
%	-\gamma_T^+\widehat{\E}^+ +		\gamma_T^-\widehat{\E}^+	+\gamma_T^-\widehat{\E}^-
%	+ \gamma_T^+\widehat{\E}^-
%	\\
%	-\gamma_T^+\widehat{\H}^+ +		\gamma_T^-\widehat{\H}^+	+\gamma_T^-\widehat{\H}^-
%	+ \gamma_T^+\widehat{\H}^-
%	\end{pmatrix}.
	\end{align}
	Subtraction of these components yields precisely the transmission conditions, namely
	\begin{align*}
	\begin{pmatrix} \gamma^+_T\widehat \E^+\inc \\ \gamma^+_T\widehat \H^+\inc \end{pmatrix}
	= 
	\begin{pmatrix}
	-\gamma_T^+\widehat{\E}^+ 	+\gamma_T^-\widehat{\E}^-
	\\
	-\gamma_T^+\widehat{\H}^+ 	+\gamma_T^-\widehat{\H}^-
	\end{pmatrix}.
	\end{align*}
	The fields $(\widehat{\E}^+,\widehat{\H}^+ )|_{\Omega^+}$ and $(\widehat{\E}^-,\widehat{\H}^- )|_{\Omega^-}$ therefore uniquely solve the transmission problem of interest.
	
	Summation of the components \eqref{eq:component-1}--\eqref{eq:component-2} yields conversely
	\begin{align*}
	\begin{pmatrix} 		\gamma_T^-\widehat{\E}^+
	+ \gamma_T^+\widehat{\E}^-
	\\
	\gamma_T^-\widehat{\H}^+
	+ \gamma_T^+\widehat{\H}^-
	\end{pmatrix} = 0 .
	\end{align*} 
	In the following, we test these equations via the anti-symmetric pairing and specific test functions. Inserting the test function $\gamma_T^- \widehat{\H}^+$ in the first component and $\gamma_T^+ \widehat \E^-$ in the second component yields
	\begin{align*}
	0 &= \Re \left[\gamma_T^- \widehat \H^+, \gamma_T^-\widehat{\E}^+
	+ \gamma_T^+\widehat{\E}^-\right]_\Gamma
	+
	\Re\left[\gamma_T^+ \widehat \E^-, \gamma_T^-\widehat{\H}^+
	+ \gamma_T^+\widehat{\H}^-\right]_\Gamma
	\\ &=
	\Re\left[\gamma_T^- \widehat \H^+, \gamma_T^-\widehat{\E}^+\right]_\Gamma
	-
	\Re\left[\gamma_T^+\widehat{\H}^-, \gamma_T^+ \widehat \E^-\right]_\Gamma.
	\end{align*} 
	As the direct consequence of \eqref{Green-EH}, we observe that $(\widehat{\E}^+,\widehat{\H}^+ )|_{\Omega^-}$ and $(\widehat{\E}^-,\widehat{\H}^- )|_{\Omega^+}$ vanish.	
	
To obtain the bound \eqref{eq:electromag-fields-bound-improved}, observe that we are now in the situation of Lemma~\ref{lem:transmission-improved}, and the claim follows together with the bounds given in Proposition~\ref{prop:bie}.
\end{proof}

\begin{remark}
In view of the $L^2$-bound of Lemma~\ref{lem:transmission-improved}, we further obtain an improved $L^2$-bound for the solution of the time-harmonic scattering problem. Under the strong passivity condition \eqref{strongly-passive} we have the bound
		%\begin{align*}
		%&\left\|\begin{pmatrix} \widehat{\E} \\ \widehat{\B} \end{pmatrix} \right\|_{L^2(\Omega^\pm)^2} \\
		%\le &\frac{\Ctrace}{\sqrt{2}}\left( \max \left(\frac{1}{\Re \eps^\pm(s)s},\frac{1}{\Re \mu^\pm (s)s}\right)\right)^{1/2}
		%\max\left(\frac{|\eps^\pm(s) s|^2+1}{\Re \varepsilon^\pm(s) s}\,,\,\frac{|\mu^\pm(s) s|^2+1}{\Re \mu^\pm(s) s}\right) \left\| \bwginc  \right\|_{\mXG^2}\\
		%\le &\frac{\Ctrace C\epsmu}{\sqrt{2}}\left( \max \left(\frac{1}{\Re \eps^\pm(s)s},\frac{1}{\Re \mu^\pm (s)s}\right)\right)^{3/2}
		%(\abs{s}^2+1)\left\| \bwginc  \right\|_{\mXG^2}.
		%\end{align*}
		%If the stronger passivity conditions holds, then we have
		\begin{align*}
			\left\|\begin{pmatrix} \widehat{\E} \\ \widehat{\B} \end{pmatrix} \right\|_{L^2(\Omega^\pm)^2} 
			\le \, C_{\Gamma,\sigma}\frac{\abs{s}^2 }{(\Re s)^{3/2}} \left\| \bwginc  \right\|_{\mXG^2}
			\qquad\text{for }\ \Re s \ge \sigma >0.
		\end{align*}
\end{remark}

%\begin{remark}
%\textcolor{blue}{In view of the $L^2$-bound of Lemma~\ref{lem:transmission-improved}, we further obtain an improved $L^2$-bound for the solution of the time-harmonic scattering problem,
%\begin{align*}
%&\left\|\begin{pmatrix} \widehat{\E} \\ \widehat{\B} \end{pmatrix} \right\|_{L^2(\Omega^\pm)^2} \\
%\le &\frac{\Ctrace}{\sqrt{2}}\left( \max \left(\frac{1}{\Re \eps^\pm(s)s},\frac{1}{\Re \mu^\pm (s)s}\right)\right)^{1/2}
%\max\left(\frac{|\eps^\pm(s) s|^2+1}{\Re \varepsilon^\pm(s) s}\,,\,\frac{|\mu^\pm(s) s|^2+1}{\Re \mu^\pm(s) s}\right) \left\| \bwginc  \right\|_{\mXG^2}\\
%\le &\frac{\Ctrace C\epsmu}{\sqrt{2}}\left( \max \left(\frac{1}{\Re \eps^\pm(s)s},\frac{1}{\Re \mu^\pm (s)s}\right)\right)^{3/2}
%(\abs{s}^2+1)\left\| \bwginc  \right\|_{\mXG^2}.
%\end{align*}
%If the stronger passivity conditions holds, then we have
%\begin{align*}
%\left\|\begin{pmatrix} \widehat{\E} \\ \widehat{\B} \end{pmatrix} \right\|_{L^2(\Omega^\pm)^2} 
%\le \frac{\Ctrace C\epsmu}{\sqrt{2}}\frac{\abs{s}^2 + 1}{(\Re s)^{3/2}} \left\| \bwginc  \right\|_{\mXG^2}.
%\end{align*}
%}
%\end{remark}

\section{The time-dependent scattering problem}
\label{sec:td-scattering}

\subsection{The time-dependent boundary integral equation}
%The time-harmonic treatment of the previous section extends to the time domain in a direct way, via the passage from the Laplace domain to the time domain  described in Section~\ref{subsec:Z}  (which follows \cite{L94}). This uses the frequency-explicit estimates of Section~\ref{section:time harmonic Maxwell} in an essential way. 
Throughout this section we assume strong passivity condition \eqref{strongly-passive}.
The time-dependent version of the boundary integral equation \eqref{eq:bie-A} is obtained by formally replacing the Laplace transform variable $s$ by the time differentiation operator~$\pt$: Given $\bginc:[0,T]\to\mXG^2$, find time-dependent boundary densities $(\bphip ,\bphim):[0,T]\to \mXG^2\times \mXG^2$ (of temporal regularity to be specified later) such that for almost every $t\in [0,T]$ we have 
\begin{align}\label{eq:bie-A-t}
\begin{pmatrix}
 \Cald_{\varepsilon^+,\mu^+}(\pt)  & \bJ \\
-\bJ  &
   \Cald_{\varepsilon^-,\mu^-}(\pt) 
\end{pmatrix}
\begin{pmatrix}
\bphip \\ \bphim
\end{pmatrix}
=
\begin{pmatrix}
-\bginc \\ \bginc
\end{pmatrix}.
\end{align}
We abbreviate this as
\begin{equation}\label{bie-A-t}
\bA(\pt) \bphi = \bg \quad\text{ with } \quad  \bphi (t) = \begin{pmatrix}
\bphip(t) \\ \bphim(t)
\end{pmatrix} \in \mXG^4, \ \  \bg (t)=  \begin{pmatrix}
-\bginc(t) \\ \bginc(t)
\end{pmatrix} \in \mXG^4.
\end{equation}
In view of the bound of Proposition~\ref{prop:bie} on the operator family $\bA(s)^{-1}$ for $\Re s>0$, the temporal convolution operator   
$$
\bA^{-1}(\pt)\bg=\mathcal{L}^{-1}\bA^{-1} * \bg
$$ 
is well-defined, and by the composition rule we have $\bA^{-1}(\pt) \bA(\pt) = \bId$ and
$\bA(\pt) \bA^{-1}(\pt) = \bId$. So we have the temporal convolution
\begin{align}\label{eq:time-Ainv}
\bphi
=\bA^{-1}(\pt) \bg
\end{align}
as the unique solution of \eqref{bie-A-t}. More precisely, with the argument given above and the convolution bound of \cite[Lemma 2.1]{L94}, we obtain the following result. Here $H^r_0(0,T; \mXG^4)$ is the space of functions on the interval $(0,T)$ taking values in $\mXG^4$ that have an extension to the real line that is in the Sobolev space $H^r(\R,\mXG^4)$.

\begin{theorem} [Well-posedness of the time-dependent boundary integral equation]
\label{prop:bie-t} Let $r\ge 0$. For $\bg\in H^{r+3}_0(0,T;\mXG^4)$,
the boundary integral equation~\eqref{bie-A-t}
%~\eqref{bie-A-t}, with the boundary operator $\A(s):\mVG\times\mXG \rightarrow \mVG'\times \mXG'$ defined by \eqref{Asoperator}, 
has a unique solution $\bphi \in H^{r+1}_0(0,T; \mXG^4)$, and
\begin{equation} \label{phi-psi-bound}
\norm{
	\bphi
}_{H^{r+1}_0(0,T;\mXG^4)}
\le %\dfrac{C_\sigma}{\mathrm{Re } \ s}
 C_T\,\norm{\bg}_{H^{r+3}_0(0,T;\mXG^4)}.
\end{equation}
Here, $C_T$ depends  on $T$ %(linearly on $T$ in the situation of Remark~\ref{rem:Res}) 
and on the boundary $\Gamma$ via norms of tangential trace operators.
\end{theorem}

\subsection{Well-posedness of the time-dependent scattering problem}

With the time-dependent boundary densities $\bphi = (\bphi^+,\bphi^-)^T$ of Theorem~\ref{prop:bie-t}, the scattered wave is obtained by the time-dependent representation formula, compactly denoted by the exterior and interior block operators $\W^{\pm}(\pt)$ via% (assuming here again $\eps\mu=1$)
\begin{align}
	\begin{pmatrix}
	\E^{\pm} \\ \H^{\pm}
	\end{pmatrix}&=
	\W^{\pm}(\pt) \bphi^{\pm}= 
	\begin{pmatrix}
-\sqrt{\frac{\mu^{\pm}}{\eps^{\pm}}}(\pt)\, \mathbfcal{S}^{\pm}\epsmu(\pt)\bvarphi^{\pm} + \mathbfcal{D}^{\pm}\epsmu(\pt)\bpsi^{\pm} \\ 
- \mathbfcal{D}^{\pm}\epsmu(\pt)\bvarphi^{\pm} - \sqrt{\frac{\eps^{\pm}}{\mu^{\pm}}}(\pt)\, \mathbfcal{S}^{\pm}\epsmu(\pt)\,\bpsi^{\pm}
	\end{pmatrix},\label{eq:time-dep-kirchhoff-EH3}
\end{align}
%\begin{align}
%	\E&=-\sqrt{\frac{\mu}{\eps}}(\pt)\, \mathbfcal{S}\epsmu(\pt)\bvarphi + \mathbfcal{D}\epsmu(\pt)\bpsi,\label{eq:time-dep-kirchhoff-E3}
%	\\
%	\H&= - \mathbfcal{D}\epsmu(\pt)\bvarphi - \sqrt{\frac{\eps}{\mu}}(\pt)\, \mathbfcal{S}\epsmu(\pt)\,\bpsi, \label{eq:time-dep-kirchhoff-H3}
%\end{align}
where we used \eqref{eq:time-harmonic-kirchhoff-E3} and \eqref{eq:time-harmonic-kirchhoff-H3}.
% together with the notation introduced in this section.
%\begin{align}
%\label{eq:time-dependent-representation-E}
%\E & = - \mathcal{S}(\pt)\bvar+\mathcal{D}(\pt)\bpsi, \\
%\label{eq:time-dependent-representation-H}
%\B & = - \mathcal{D}(\pt) \bvar - \mathcal{S}(\pt) \bpsi .
%\end{align}

We now give the well-posedness result for the time-dependent scattering problem,
which follows from the time-harmonic well-posedness result Proposition~\ref{prop:wp-time-harm-scattering}.

\begin{theorem}[Well-posedness of the time-dependent scattering problem]
\label{th:time-dependent-well-posedness} \ Consider the time-dependent scattering problem \eqref{eq:td-scattering-domain-2} equipped with \eqref{eq:td-transmission-conditions} %(with the normalization $\eps\mu=1$)
%	under the generalized impedance boundary condition \eqref{gibc-weak},
% with $\Z(s)$ satisfying conditions \eqref{eq:pol_bound}--\eqref{eq:positive_type} with $\kappa\le 1$ and 
and $\bginc \in H^{r+3}_0(0,T;\mXG^2)$ for some arbitrary $r\ge 0$.

	(a) This problem has a unique solution
	$$({\E^{\pm}},{\B^{\pm}}) \in H^r_0(0,T;\H(\curl,\Omega^{\pm})^2) \cap H^{r+1}_0(0,T;(\bL^2(\Omega^{\pm}))^2)
	$$
given by the
	representation  formulas \eqref{eq:time-dep-kirchhoff-EH3}. The tangential traces
 are uniquely determined by the  solution of the system of boundary integral equations of Theorem~\ref{prop:bie-t},
	$$(\bvar^{\pm},\bpsi^{\pm})=(\gaT{\H^{\pm}},-\gaT {\E^{\pm}})\in H^{r+1}_0(0,T;\mXG \times \mXG).$$

	(b) The electromagnetic fields are bounded by
		\begin{align*}
		\| {\E^{\pm}} \|_{H^r_0(0,T;\H(\curl,\Omega^{\pm}))} + \| {\B^{\pm}} \|_{H^r_0(0,T;\H(\curl,\Omega^{\pm}))}
	\le %\dfrac{C_\sigma}{\mathrm{Re } \ s}
	C_T \|\bginc\|_{H^{r+3}_0(0,T;\mXG^2)},
	\end{align*}
	and the same bound is valid for the $H^{r+1}_0(0,T;(\bL^2(\Omega^{\pm}))^2)$ norms.
 Here, $C_T$ depends polynomially on $T$, on the boundary $\Gamma$ via norms of tangential trace operators, and on the bounds of the frequency dependent material parameters $\eps,\mu$.
% but is independent of $\eps$ and $\mu$ with $\eps\mu=1$ 
% and, in the case of the impedance operators \eqref{eq:thin_layer_ord1}--\eqref{eq:gibc_absorbing_ord2}, independent of the small parameter $\delta$.
\end{theorem}

\begin{proof}
	The proof is identical to \cite[Thm.~4.2]{NKL22}.
\end{proof}

%%%%%

\section{Semi-discretization in time by Runge--Kutta convolution quadrature}
\label{section:time semi-discrete}

%\subsection{BDF convolution quadratures}
%We approximate the convolution $K(\partial_t)g$, cf.~\eqref{eq:Heaviside notation}, using \emph{convolution quadratures} (CQ), for more details see, e.g., \cite{L88,L94}. For a fixed time step size $\tau > 0$, we denote the approximation
%\begin{align*}
%	\left( K(\pttau)g\right) (t):= \sum_{j\ge 0} w_j g(t-j\tau).
%\end{align*}
%The quadrature weights are defined as the coefficients of the generating power series
%\begin{align*}
%	\sum_{j=0}^{\infty}w_j \zeta^j:=K\left(\dfrac{\delta(\zeta)}{\tau}\right), \quad \text{for} \quad \left| \zeta \right|
%	\quad \text{small.}
%\end{align*}
%We write $\delta(\zeta)=\sum_{j=0}^{\infty}\delta_j\zeta^j$ is the quotient of generating polynomials of a $p$-step linear multistep method
%$\sum_{j=0}^{p} \alpha_j y^{n-j} = \tau \sum_{j=0}^p \beta_j f_{n-j}$, i.e.
%\begin{align*}
%	\delta(\zeta) = \dfrac{\alpha_0+\alpha_1 \zeta + \dotsb + \alpha_k\zeta^k}{\beta_0+\beta_1\zeta + \dotsb + \beta_k \zeta^k}.
%\end{align*}
%
%
%The convolutions in time of system \eqref{Aoperator} are now discretized, using the convolution quadrature based on the backward difference formulae (BDF) of order 1 and 2, which have generating polynomials
%\begin{alignat*}{3}
%	& \text{$1$-step BDF:} & \qquad & \beta(\zeta) = 1, \quad \alpha(\zeta) = 1 - \zeta , \\
%	& \text{$2$-step BDF:} & \qquad & \beta(\zeta) = 1, \quad \alpha(\zeta) = \frac{3}{2} - 2 \zeta + \frac{1}{2} \zeta^2 .
%\end{alignat*}
\subsection{Recap: Runge--Kutta convolution quadrature}

To approximate the omnipresent temporal convolutions $K(\partial_t)g$, we will employ the convolution quadrature method based on Runge--Kutta time stepping schemes. In order to introduce the notation, we recall an $m$-stage implicit Runge--Kutta  discretization of the initial value problem $y' = f(t,y)$, $y(0) = y_0$;  see \cite{HairerWannerII}. For some constant time step $\tau > 0$, the approximations $y^n$ to $y(t_n)$ at time $t_n = n \tau$, and the internal stages $Y^{ni}$ approximating $y(t_n + c_i \tau)$, are computed by solving the system
\begin{equation*}
\begin{aligned}
Y^{ni} &= y^n + \tau \sum_{j = 1}^m a_{ij} f(t_n+c_jh,Y^{nj}), \qquad i =
1,\dotsc,m,\\
y^{n+1} & = y^n + \tau \sum_{j = 1}^m b_j f(t_n+c_jh,Y^{nj}) .
\end{aligned}
\end{equation*}
The method is uniquely defined by the Butcher-tableau, which collects its coefficients
\begin{equation*}
\mathscr{A} = (a_{ij})_{i,j = 1}^m , \quad \cqb = (b_1,\dotsc,b_m)^T,
\quad \text{and} \quad \cqc = (c_1,\dotsc,c_m)^T .
\end{equation*}
The stability function of the Runge--Kutta method is given by $R(z) = 1 + z b^T (\cqI - z \mathscr{A})^{-1} \bone$, where $\bone = (1,1,\dotsc,1)^T \in \R^m$. We always assume that $\mathscr{A}$ is invertible.

Runge--Kutta methods can be used to construct convolution quadrature methods.
Such methods were first introduced in \cite{LubichOstermann_RKcq} in the context of parabolic problems and were studied for wave propagation problems in \cite{BLM11} and subsequently, e.g., in \cite{BanjaiKachanovska,BanjaiLubich2019,BanjaiMessnerSchanz,BanjaiRieder}. Runge--Kutta convolution quadrature was studied  for the numerical solution of some exterior Maxwell problems in \cite{BBSV13,ChenMonkWangWeile,NKL22} and of an eddy current problem with an impedance boundary condition in \cite{HiptmairLopezFernandezPaganini}. For wave problems, Runge--Kutta convolution quadrature methods such as those based on the Radau IIA methods, see \cite[Section~IV.5]{HairerWannerII}, often enjoy more favourable properties than their BDF-based counterparts, which are more dissipative and cannot exceed order~2 but are easier to understand and slightly easier to implement.

Let $\cqK(s):\cqX\to \cqY$, $\Re s \ge \sigma_0>0$,  be an analytic family of linear operators between Banach spaces $\cqX$ and $\cqY$, satisfying
the bound, for some exponents $\kappa\in\R$ and $\nu\ge 0$,
\begin{equation}\label{KLM-bound}
\| \cqK(s) \|_{\cqY\leftarrow \cqX} \le M_\sigma \frac{|s|^\kappa}{(\Re s)^\nu}, \qquad \Re s \ge \sigma> \sigma_0.
\end{equation}
This yields a convolution operator $\cqK(\pt):H^{r+\kappa}_0(0,T;\cqX) \to \cqH^{r}_0(0,T;\cqY)$  for arbitrary real $r$. For functions $\cqg:[0,T]\to \cqX$ that are sufficiently regular (together with their extension by 0 to the negative real half-axis $t<0$),
we wish to approximate the convolution $(\cqK(\pt)\cqg)(t)$ at discrete times $t_n=n\tau$ with a step size $\tau>0$, using a discrete convolution.

To construct the convolution quadrature weights, we use the \emph{Runge--Kutta differentiation symbol}
\begin{equation}
\label{eq:Delta}
\Delta(\zeta) = \Bigl(\mathscr{A}+\frac\zeta{1-\zeta}\bone \cqb^T\Bigr)^{-1} \in \C^{m \times m}, \qquad
\zeta\in\C \hbox{ with } |\zeta|<1.
\end{equation}
This is well-defined for $|\zeta|<1$ if %$\mathscr{A}$ is invertible and if
$R(\infty)=1-b^T\mathscr{A}^{-1}\bone$ satisfies $|R(\infty)|\le 1$, as is seen from the Sherman--Woodbury formula. Moreover, for A-stable Runge--Kutta methods (e.g. the Radau IIA methods), the eigenvalues of the matrices $\Delta(\zeta)$ have positive real part for $|\zeta|<1$~\cite[Lemma 3]{BLM11}.
%In fact, the Sherman--Morrison formula then yields
%\begin{equation*}
%	\Delta(\zeta) = \mathscr{A}^{-1} -\frac{\zeta}{1-R(\infty)\zeta} \mathscr{A}^{-1} \bone b^T \mathscr{A}^{-1}.
%\end{equation*}

To formulate the Runge--Kutta convolution quadrature for $\cqK(\partial_t )\cqg$, we replace the complex argument $s$ in $\cqK(s)$ by the matrix $\Delta(\zeta)/\tau$ and expand
\begin{equation}\label{rkcq-weights}
\cqK\Bigl(\frac{\Delta(\zeta)}\tau \Bigr) = \sum_{n=0}^\infty {\boldsymb W}_n(\cqK) \zeta^n.
\end{equation}
The operators ${\cqW}_n(\cqK):\cqX^m \to \cqY^m$ are used as the convolution quadrature ``weights''.
For the discrete convolution of these operators with a sequence $\cqg=(\cqg^n)$ with $\cqg^n=(\cqg^n_i)_{i=1}^m\in \cqX^m$
we use the notation
\begin{equation}\label{rkcq}
\bigl(\cqK(\pttau) \boldsymb g \bigr)^n = \sum_{j=0}^n {\boldsymb W}_{n-j}(\cqK) \boldsymb g^j \in \cqY^m.
\end{equation}
Given a function $\boldsymb g:[0,T]\to \cqX$, we use this notation for the vectors $\cqg^n = \bigl(\cqg(t_n+c_i\tau)\bigr)_{i=1}^m$ of values of $\cqg$.
The $i$-th component of the vector $\bigl(\cqK(\pttau) \boldsymb g \bigr)^n$ is then an approximation to $\bigl(\cqK(\partial_t)\cqg\bigr)(t_n+c_i\tau)$; see \cite[Theorem 4.2]{BanjaiLubich2019}.

In particular, if $c_m = 1$, as is the case with Radau IIA methods,
the continuous convolution at $t_{n}$ is approximated by the $m$-th, i.e.~last component of the $m$-vector \eqref{rkcq} for $n-1$:
\begin{equation*}
\bigl(\cqK(\partial_t) \cqg \bigr)(t_{n}) \approx   \Bigl[\bigl(\cqK( \pttau) \boldsymb g \bigr)^{n-1}\Bigr]_m \in \cqY. 
\end{equation*}
This discretization \eqref{rkcq} inherits the composition rule \eqref{comp-rule}: For two analytic families of operators $\cqK(s)$ and $\cqL(s)$ mapping into compatible spaces, the convolution quadrature discretization satisfies 
\begin{equation}\label{comp-rule-cq}
\cqK(\pttau)\cqL(\pttau) \boldsymb g = ( \cqK\cqL)(\pttau)  \boldsymb g;
\end{equation}
see e.g. \cite[Equation (3.5)]{L94}.

%where ${\boldsymb e}_m = (0, \dotsc, 0, 1)^T \in \R^m$ is the $m$-th unit vector.
%
%An essential property is that the composition rule~\eqref{comp-rule} is preserved under this discretization: for two such operator families $\cqK(s)$ and $\cqL(s)$ that map to compatible spaces, we have
%\begin{equation}\label{comp-rule-tau}
%\cqK(\pttau)\cqL(\pttau)\cqg = (\cqK\cqL)(\pttau)\cqg.
%\end{equation}

The following error bound for Runge--Kutta convolution quadrature from~\cite{BLM11}, here directly stated for the Radau IIA methods \cite[Section~IV.5]{HairerWannerII} and transferred to a Banach space setting, will be the basis for our error bounds of the  time discretization.

\begin{lemma}[{\cite[Theorem~3]{BLM11}}]
	\label{lem:RK-CQ}
	Let $\cqK(s):\cqX\to \cqY$, $\Re s > \sigma_0\ge0$,  be an analytic family of linear operators between Banach spaces $\cqX$ and $\cqY$ satisfying
the bound \eqref{KLM-bound} with exponents $\kappa$ and $\nu$.
	Consider the Runge--Kutta convolution quadrature based on the Radau IIA method with $m$ stages. Let $1\le q\le m$ (the most interesting case is $q=m$) and $r>\max(2q-1+\kappa,2q-1,q+1)$. Let $\cqg \in \cqC^r([0,T],\cqX)$ satisfy $\cqg(0)=\cqg'(0)=...=\cqg^{(r-1)}(0)=0$. Then, the following error bound holds at $t_n=n\tau\in[0,T]$:
	\begin{align*}
	&\norm{{ \Bigl[\bigl(\cqK( \pttau) \boldsymb g \bigr)^{n-1}\Bigr]_m-(\cqK(\pt)\cqg)(t_{n})} }_{\cqY}
	\\
	&\quad\quad\quad \le
	C\, M_{1/T}\,\tau^{\min(2q-1,q+1-\kappa+\nu)}
	\left(\|{\cqg^{(r)}(0)}\|_{\cqX}+\int_0^t\|{\cqg^{(r+1)}(t')}\|_{\cqX} \,\mathrm{d}t'
	\right).
	\end{align*}
	The constant C is independent of $\tau$ and $\cqg$ and $M_\sigma$ of \eqref{KLM-bound}, but depends on the exponents $\kappa$ and $\nu$ in \eqref{KLM-bound} and on the final time $T$.
\end{lemma}

%\begin{lemma}[{\cite[Theorem~3]{BLM11}}]
%	\label{lem:RK-CQ}
%	Let $K(s):X\to Y$, $\Re s > \sigma_0\ge0$,  be an analytic family of linear operators between Banach spaces $X$ and $Y$ satisfying
%the bound \eqref{KLM-bound} with exponents $\kappa$ and $\nu$.
%	Consider the Runge--Kutta convolution quadrature based on the Radau IIA method with $m$ stages. Let $r>\max(2m-1+\kappa,2m-1,m+1)$ and $g \in C^r([0,T],X)$ satisfy $g(0)=g'(0)=...=g^{(r-1)}(0)=0$. Then, the following error bound holds at $t_n=n\tau\in[0,T]$:
%	\begin{align*}
%	&\|{(K(\pttau) \boldsymb g)^{n-1}_m-(K(\pt)g)(t_{n})}\|_Y
%	\\
%	&\quad\quad\quad \le
%	C\, M_{1/T}\,\tau^{\min(2m-1,m+1-\kappa+\nu)}
%	\left(\|{g^{(r)}(0)}\|_X+\int_0^t\|{g^{(r+1)}(t')}\|_X \,\mathrm{d}t'
%	\right).
%	\end{align*}
%	The constant C is independent of $\tau$ and $g$ and $M_\sigma$ of \eqref{KLM-bound}, but depends on the exponents $\kappa$ and $\nu$ in \eqref{KLM-bound} and on the final time $T$.
%\end{lemma}
%

\subsection{Convolution quadrature for the scattering problem}
Throughout the following sections, we assume strong passivity \eqref{strongly-passive} for the frequency dependent parameters $\eps(s), \mu(s)$. Applying a Runge--Kutta based convolution quadrature discretization to the temporal convolution equation \eqref{bie-A-t} reads
\begin{align}
		\label{eq:AinvoperatorCQ}
		\A(\pttau) 
		\bphi_\tau
		=
		\bg,
		\qquad\text{or equivalently},\qquad
		\bphi_\tau
		=
		\A^{-1}(\pttau)\, 
		\bg,
\end{align}
%to the semi-discretization in time of the time-dependent boundary integral equation \eqref{bie-A-t} yields the discrete convolution equation
%\textcolor{blue}{\begin{align}
%	\label{bie-A-tau}
%	\A(\pttau)\bphi^\tau
%	=
%	\textcolor{blue}{\bg},
%	% \begin{pmatrix}
%	%\ginc \\
%	%0
%	%\end{pmatrix}.
%	\end{align}}
%\begin{align}
%\label{bie-A-tau}
%\A(\pttau)\begin{pmatrix}
%\bvar^\tau\\
%\bpsi^\tau
%\end{pmatrix}
%=
%\textcolor{blue}{\bg},
%% \begin{pmatrix}
%%\ginc \\
%%0
%%\end{pmatrix}.
%\end{align}
where $\bphi$ and $\bg$ are defined in \eqref{bie-A-t}, and the equivalence of the two formulations is a consequence of the discrete composition rule \eqref{comp-rule-cq}.
%\begin{align*}
%%\label{AinvoperatorCQ}
%\begin{pmatrix}
%\bvar^\tau\\
%\bpsi^\tau
%\end{pmatrix}
%=
%\A^{-1}(\pttau)\, 
%\textcolor{blue}{\bg.}
%%\begin{pmatrix}
%%\ginc \\
%%0
%%\end{pmatrix}.
%\end{align*}
 This formulation, which is equivalent to discretizing the boundary integral equation \eqref{bie-A-t} with the convolution quadrature method and inverting the quadrature weights, interprets the solution of the discretized boundary integral equation as a forward convolution quadrature. The error of this formulation is then bounded by the error estimate of Lemma~\ref{lem:RK-CQ}, through the bound of $\bA^{-1}(s)$ given in Proposition~\ref{prop:bie}. This argument for the stability of the formulation and the resulting path to error estimates originates from~\cite{L94}, for a time-dependent boundary integral equation derived in the context of an acoustic problem.

The time discretizations of the electromagnetic fields are then obtained by applying the convolution quadrature to the representation formulas \eqref{eq:time-dep-kirchhoff-EH3}
with $\bphi^\pm_\tau = (\bvar^\pm_\tau, \bpsi^\pm_\tau)$:
\begin{align}
	\begin{pmatrix}
		\E_\tau^{\pm} \\ \H_\tau^{\pm}
	\end{pmatrix}&=
	\W_{\epsmu}^{\pm}(\pt^\tau) \bphi_\tau^{\pm}
	= 
%	\begin{pmatrix}
%	-\sqrt{\frac{\mu^{\pm}}{\eps^{\pm}}}(\pt^\tau)\, \mathbfcal{S}^{\pm}\epsmu(\pt^\tau) & \mathbfcal{D}^{\pm}\epsmu(\pt^\tau) \\ 
%	- \mathbfcal{D}^{\pm}\epsmu(\pt^\tau) & - \sqrt{\frac{\eps^{\pm}}{\mu^{\pm}}}(\pt^\tau)\, \mathbfcal{S}^{\pm}\epsmu(\pt^\tau)\,
%\end{pmatrix}
%\begin{pmatrix}
%\bvarphi_\tau^{\pm}\\	\bpsi_\tau^{\pm}
%\end{pmatrix} .
%\label{eq:time-dep-kirchhoff-EH3-tau}
	\begin{pmatrix}
		-\sqrt{\frac{\mu^{\pm}}{\eps^{\pm}}}(\pt^\tau)\, \mathbfcal{S}^{\pm}\epsmu(\pt^\tau)\bvarphi_\tau^{\pm} + \mathbfcal{D}^{\pm}\epsmu(\pt^\tau)\bpsi_\tau^{\pm} \\ 
		- \mathbfcal{D}^{\pm}\epsmu(\pt^\tau)\bvarphi_\tau^{\pm} - \sqrt{\frac{\eps^{\pm}}{\mu^{\pm}}}(\pt^\tau)\, \mathbfcal{S}^{\pm}\epsmu(\pt^\tau)\,\bpsi_\tau^{\pm}
	\end{pmatrix}.\label{eq:time-dep-kirchhoff-EH3-tau}
\end{align}

%\begin{align}
%\label{eq:time-dependent-representation-E-tau}
%\E^\tau & = - \sqrt{\frac{\mu}{\eps}} (\pttau)\,\mathbfcal{S}_{\epsmu}(\pttau)\bvar^\tau+\mathbfcal{D}_{\epsmu}(\pttau)\bpsi^\tau, \\
%\label{eq:time-dependent-representation-H-tau}
% \B^\tau & = - \mathbfcal{D}_{\epsmu}(\pttau) \bvar^\tau - \sqrt{\frac{\eps}{\mu}}(\pttau) %-\frac{\sqrt{\mu(\pttau)}}{\sqrt{\eps(\pttau)}}
%  \mathbfcal{S}_{\epsmu}(\pttau) \bpsi^\tau .
%\end{align}
By the discrete composition rule \eqref{comp-rule-cq}, this is the convolution quadrature discretization of the composed operator
\begin{equation}\label{U-t}
\begin{pmatrix} \E^{\pm}_\tau \\ \B^{\pm}_\tau
\end{pmatrix}
= \U_{\epsmu}^{\pm}(\pttau) \bginc
\qquad \text{of}\qquad
\begin{pmatrix} \E^{\pm} \\  \B^{\pm}
\end{pmatrix}
= \U_{\epsmu}^{\pm}(\pt) \bginc,
\end{equation}
where we have by Theorem~\ref{th:time-dependent-well-posedness} that 
%$$
%\textcolor{blue}{\U\epsmu(s)} = \begin{pmatrix}
%-\mathcal{S}\epsmu(s) & \mathcal{D}\epsmu(s)
%\\
%-\mathcal{D}\epsmu(s) & -\mathcal{S}\epsmu(s)
%\end{pmatrix} \bAinvs \begin{pmatrix} \Id \\ 0 \end{pmatrix}
%:
%\mXG^2 \to \H(\curl,\Omega)^2,
%$$
$	\U_{\epsmu}^{\pm}(s) :
	\mXG^2 \to \H(\curl,\Omega^{\pm})^2 $
is given by 
\begin{equation}\label{eq:auxiliary-operators}
\U_{\epsmu}^{\pm}(s) = \W_{\epsmu}^{\pm}(s)
 \mathbfcal P^{\pm}\bAinvs \mathbfcal N ,
%:
%\mXG^2 \to \H(\curl,\Omega)^2,
\end{equation}
with the auxiliary maps $\mathbfcal P^{\pm} \colon \mXG^4\rightarrow \mXG^2$ projecting on the exterior and interior boundary densities respectively and $\mathbfcal N \colon \mXG^2\rightarrow \mXG^4$ expanding the terms of the incident wave via
\begin{align*}
\mathbfcal P^+ = 
\begin{pmatrix}
	\bId & 0 
\end{pmatrix}, \quad \mathbfcal P^-= \begin{pmatrix}0 & \bId \end{pmatrix} \quad\text{ and } \quad \mathbfcal N =    \begin{pmatrix} -\bId & \bId \end{pmatrix}^{\top} .
\end{align*}
Under the stronger passivity condition, we then have by Proposition~\ref{prop:wp-time-harm-scattering} the bound
%for which the bound
\begin{equation}\label{eq:U-Hcurl-bound}
\| \U_{\epsmu}^{\pm}(s) \|_{\H(\curl,\Omega^{\pm})^2 \leftarrow \mXG^2} \le C_{\sigma} \dfrac{\abs{s}^3}{(\Re s)^{3/2}}, \quad\text{ for }\  \Re s \ge \sigma > 0.
\end{equation}

%\color{blue}
Moreover, away from the boundary on $\Omega^{\pm}_d=\{ \x\in \Omega^{\pm}\,:\,\dist(\x,\Gamma)> d\}$ with $d >0$, bounds that decay exponentially with the real part of $s$ hold. The following lemma is a direct consequence of \cite[Lemma~3.8]{NKL22} and of Lemma~\ref{lem:eps-mu-s} to obtain the following parameter-dependent bound.

	\begin{lemma}\label{lem:bound-away-boundary-deriv} 
		Under the strong passivity condition~\eqref{strongly-passive}, we have the following bounds at $\bx\in \R^3\setminus \Gamma$  with $d = \dist(\x,\Gamma) > 0$ and for $\Re s \ge \sigma >0$
	\begin{align*}
	\abs{\big(\mathbfcal{S}\epsmu(s)\bvarphi\big)(x)}  &\le C_\sigma\abs{s}^{2}\e^{-dc\Re s}\norm{\bvarphi}_\mXG,\\
	\abs{\big(\mathbfcal{D}\epsmu(s)\bvarphi\big)(x)} &\le C_\sigma\ \abs{s}^{2}\e^{-dc\Re s}\norm{\bvarphi}_\mXG.
	\end{align*}
	for all $\bvarphi \in \mXG$.
	\end{lemma}

% \begin{remark}
%	All the presented examples of Section~\ref{sec:rml_examples}, identified by the functions $\varepsilon(s)$ and $\mu(s)$, satisfy the following additional estimate: 
%	
%	\begin{equation}\label{eq:additional-assumption}
%		\Re\bigl(s \sqrt{\mu(s)\eps(s)}\bigr) \ge c \Re s ,
%	\end{equation}
%	for some constant $c>0$.  Under this assumption, the exponential decay rate bounding the potential operators simplifies to $e^{-dc\Re s}$.
%\end{remark}
%\color{black}

Combining Lemma~\ref{lem:bound-away-boundary-deriv} and Theorem~\ref{prop:bie} yields, under the assumption of strong passivity \eqref{strongly-passive} and using \eqref{eps-mu-factor-bounds},
%$$
%\| \U(s) \|_{(\bC^1(\overline\Omega_d)^3)^2\leftarrow \mXG^2} +\| \U(s) \|_{\H(\curl,\Omega_d)^2 \leftarrow \mXG^2} \le C_\sigma \abs{s}^5 e^{-d \Re s} ,
%$$
\begin{align}\label{eq:U-pointwise-bound}
\| \U^{\pm}\epsmu(s) \|_{(\bC^1(\overline\Omega^{\pm}_d)^3)^2\leftarrow \mXG^2}
\le C_{\sigma}\frac{\abs{s}^{1/2}}{\left(\Re s\right)^{1/2}}\frac{\abs{s}^{4}}{\Re s}
e^{-d c\Re s}\,
\norm{\bwginc}_{\mXG^2},
\end{align}
for $\Re s\ge \sigma>0$. 
%We note that this estimate is obtained by bounding the factors appearing in the representation formulas   \eqref{eq:time-harmonic-kirchhoff-E3}--\eqref{eq:time-harmonic-kirchhoff-H3} via
%\begin{align*}
%\abs{\dfrac{\varepsilon (s)}{\mu(s)}}
%\le \dfrac{  \abs{\varepsilon (s) s }}{\Re \mu(s)s}
%\le C _\sigma \dfrac{\abs{s}}{\Re s },
%\end{align*}
%and the corresponding estimate for its reciprocal.
The $\bC^1(\overline\Omega^{\pm}_d)$-norm denotes the maximum norm on continuously differentiable functions and their derivatives on the closure of the domains $\Omega^{\pm}_d$ respectively. 

For the sake of brevity, we omit a formulation of error bounds for the temporal semi-discretization and continue with a full discretization for the boundary integral equation. 

%\begin{remark}
%	Discrete fields $(\E^\tau,\B^\tau)$ generated through the discretized representation formulas \eqref{eq:time-dependent-representation-E-tau}--\eqref{eq:time-dependent-representation-H-tau} are, just as their continuous counterparts, divergence-free. This follows from the observation that the generating functions $\sum_{n\ge 0} \E^n \zeta^n, \sum_{n\ge 0} \B^n \zeta^n$ are divergence-free for $\abs{\zeta}<1$.
%\end{remark}

\section{Full discretization}
\label{sec:full}

Finally, we combine a convolution quadrature time discretization of \eqref{bie-A-t} with a spatial Galerkin approximation of the boundary operators, based on a boundary element space 
% $\V_h\subset \mVG$ and 
 $\X_h \subset \mXG$, which corresponds to a family of triangulations with decreasing mesh width $h\rightarrow 0$. Throughout this paper, we use Raviart--Thomas boundary elements of order $k\ge 0$, which are defined on the unit triangle $\widehat{K}$ by
\begin{align*}
\text{RT}_k(\widehat{K}) = \left\{\x\mapsto \p_1(\x)+p_2(\x) \x \,:\  \p_1\in P_k(\widehat{K})^2\!,\  p_2\in P_k(\widehat{K})\right\},
\end{align*}
where $P_k(\widehat{K})$ contains all polynomials of degree $k$ on $\widehat{K}$. 
The definition is then extended to arbitrary triangles in the standard way via pull-back to the reference element. Details are found in the original paper \cite{RT77}.

The following approximation result holds with respect to the $\mXG$-norm; see also the original references \cite[Section~III.3.3]{BreF91} and \cite{BC03}.
Here we use the same notation $\H^p_\times(\Gamma)=\gaT \H^{p+1/2}(\Omega)$ as in \cite{BH03}.

\begin{lemma}[{\cite[Theorem 14]{BH03}}]
	\label{lem:RT}
Let $\X_h$ be the $k$-th order Raviart--Thomas boundary element space on $\Gamma$. There exists a constant $C$, such that the best-approximation error of any $\bxi\in\mXG\cap \H^{k+1}_\times(\Gamma)$
%and $\bupsilon\in\mVG\cap \H^{k+1}_\times(\Gamma)$, 
%with the space $\mVG$ of Lemma~\ref{lem:EN} or Lemma~\ref{lem:absorb}, 
 is bounded by
\begin{align*}
&\inf_{\bxi_h\in \X_h} \| \bxi_h - \bxi \|_{\mXG} \le C h^{k+3/2} \|\bxi \|_{\H^{k+1}_\times(\Gamma)}.
%\\
%&\inf_{\bupsilon_h \in\V_h} \| \bupsilon_h - \bupsilon\|_{\mVG} \le C h^{k+1} \|\bupsilon \|_{\H^{k+1}_\times(\Gamma)}.
\end{align*}
\end{lemma}
%\begin{remark}
%\label{remark:bestappprox}
%We would have expected that the best-approximation error bound in the $\mVG$-norm is $O(h^{k+3/2}+ \delta^{1/2} h^{k+1})$, in analogy to the situation for acoustic  generalized impedance boundary conditions~\cite{BLN20}.
%This would, however, require proving the $\mVG$-norm stability of the projection of \cite{BC03} from $\mXG$ to $\X_h$ that was used to show the best-approximation estimate in $\mXG$. If at all possible, this is in any case beyond the scope of this paper.
%\end{remark}
\noindent The full discretization of boundary integral equation \eqref{bie-A-t} on $\X_h^4$ then reads
\begin{align}\label{bie-A-tau-h}
\left[
\bxi_h\, ,\,
\A(\pttau)\bphi_h^\tau \right]_\Gamma=\left[\ \bxi_h
\,,\,\bg\inc\right]_\Gamma \quad \quad \forall \,\bxi_h \in (\X_h^4)^m.
\end{align}
This formulation determines the approximate boundary densities, by
\begin{equation}
\left(\bphi_h^\tau\right)^n = \left((\bvarphi^{+}_{\tau,h})^n, (\bpsi^{+}_{\tau,h})^n,(\bvarphi^{-}_{\tau,h})^n, (\bpsi_{\tau,h}^{+})^n\right)^\top \in \X_h^4,
\end{equation}
where
 $\bvar^{\pm}_{\tau,h}=\bigl((\bvar^{\pm}_{\tau,h})^n\bigr)$
with $(\bvar^{\pm}_{\tau,h})^n$ $=\bigl( (\bvar^{\pm}_{\tau,h})^n_i \bigr)_{i=1}^m\in \X_h^m$ . The electric densities $\bpsi^{\pm}_{\tau,h}$ are defined in the same way. %$\bpsi^{\pm}_{\tau,h}=\bigl((\bpsi^{\pm}_{\tau,h})^n\bigr)$
%with $(\bpsi^{\pm}_{\tau,h})^n$ $=\bigl( (\bpsi^{\pm}_{\tau,h})^n_i %\bigr)_{i=1}^m\in \X_h^m$.  , 
The approximations to the electromagnetic fields are obtained via the time-discrete representation formulas on the interior domain $\Omega^{-}$ and the exterior domain $\Omega^+$: 
\begin{align}
\label{eq:time-dependent-representation-E-tau-h}
\E^{\pm}_{\tau,h} & = - \sqrt{\frac{\mu}{\eps}}(\pttau)\, \mathbfcal{S}\epsmu(\pttau)\bvar^{\pm}_{\tau,h}+\mathbfcal{D}\epsmu(\pttau)\bpsi^{\pm}_{\tau,h}, \\
\label{eq:time-dependent-representation-H-tau-h}
\B^{\pm}_{\tau,h} & = - \mathbfcal{D}\epsmu(\pttau) \bvar^{\pm}_{\tau,h} - \sqrt{\frac{\eps}{\mu}}(\pttau)\,\mathbfcal{S}\epsmu(\pttau) \bpsi^{\pm}_{\tau,h} .
\end{align}
These fully discrete approximations satisfy the following error bounds, obtained under regularity assumptions that are presumably stronger than necessary.

\begin{theorem} [Error bound of the  full discretization] \label{th:error-bounds}
Consider the setting and assumptions of Theorem~\ref{th:time-dependent-well-posedness} and further let $\varepsilon^{\pm}$ and $\mu^{\pm}$ satisfy the strong passivity~\eqref{strongly-passive}.

Consider the fully discrete scheme \eqref{bie-A-tau-h} and the temporally discrete representation formulas  \eqref{eq:time-dependent-representation-E-tau-h}--\eqref{eq:time-dependent-representation-H-tau-h}, where the $m-$stage Radau IIA convolution quadrature discretization and $k$-th order Raviart--Thomas boundary element discretization have been employed as described in the previous sections.
%We employ Runge--Kutta convolution quadrature based on the Radau IIA method with $m\ge 2$ stages used for  the time discreti\-zation of the boundary integral equation \eqref{bie-A-tau-h} and the representation formulas  \eqref{eq:time-dependent-representation-E-tau-h}--\eqref{eq:time-dependent-representation-H-tau-h}. 
%%\eqref{eq:time-dep-kirchhoff-E3}--\eqref{eq:time-dep-kirchhoff-H3},\eqref{eq:time-dependent-representation-E}--\eqref{eq:time-dependent-representation-H}, 
%The space discretization is conducted as described, by applying the Galerkin method with Raviart--Thomas boundary elements of order $k$ to the boundary integral equation~\eqref{bie-A-t}.
\\
For $r>2m+3$ we assume the incoming waves to satisfy $\bginc \in \bC^r([0,T],\mXG^4)$. Moreover, we assume $\bginc$ to vanish at $t=0$ together with its first $r-1$ time derivatives. Furthermore, it is
assumed that the solution $\bphi$ of the boundary integral equation~\eqref{bie-A-t} is at least in
$\bC^{10}([0,T],\H^{k+1}_\times(\Gamma)^2)$, vanishing at $t=0$ together with its time derivatives.
\\
Then, the approximations to the electromagnetic fields at time $t_n$, both in the interior and the exterior domain,
$$\left( \E^{\pm}_{\tau,h} \right)^n=\bigl[(\E^{\pm}_{\tau,h} )^{n-1}\bigr]_m \quad \text{and} \quad \left( \H^{\pm}_{\tau,h} \right)^n = \bigl[(\B^{\pm}_{\tau,h})^{n-1}\bigr]_m\, ,
$$ 
satisfy the following error bound of order $m-1/2$ in time and order $k+3/2$ in space at $t_n=n\tau\in[0,T]$:
\begin{align*}
	\norm{\left(\E^{\pm}_{\tau,h} \right)^n- \E(t_n) }_{\H(\curl,\Omega^{\pm})} +
	\norm{\left(\H^{\pm}_{\tau,h} \right)^n- \H(t_n) }_{\H(\curl,\Omega^{\pm})}
%\norm{ \begin{pmatrix} \left( \E^{\pm}_{\tau,h} \right)^n- \E(t_n) \\ \B^n_h - \B(t_n)
%\end{pmatrix} } _{\H(\curl,\Omega)^2}
\le C \bigl(\tau^{m-1/2}+h^{k+3/2}\bigr).
\end{align*}
	For $r>2m+4$, we obtain the full order $2m-1$ in time away from the interface $\Gamma$, on the domains $\Omega^{\pm}_d=\{ \x\in \Omega\,:\,\dist(\x,\Gamma)> d\}$ with $d >0$, which reads 
\begin{align*}
\norm{\left(\E^{\pm}_{\tau,h} \right)^n- \E(t_n) }_{\bC^1(\overline\Omega^{\pm}_d)^3} +
\norm{\left(\H^{\pm}_{\tau,h} \right)^n- \H(t_n) }_{\bC^1(\overline\Omega^{\pm}_d)^3}
%\norm{ \begin{pmatrix} \E^n_h - \E(t_n) \\ \B^n_h - \B(t_n)
%\end{pmatrix} }_{\big(\bC^1(\overline\Omega^{\pm}_d)^3\big)^2}
\le C_d  \bigl(\tau^{2m-1}+h^{k+3/2}\bigr).
\end{align*}
The constants $C$ and $C_d$ are independent of $n$, $\tau$ and $h$, but depend on the final time~$T$ and on the regularity of $\bginc$ and $(\bvar,\bpsi)$ as stated. $C_d$ additionally depends on the distance $d$. 
\end{theorem}

%\begin{proof}
%    The proof is a combination of the best  by \cite[Proof of Theorem 3.5]{NKL22}
%\end{proof}

\begin{proof} The proof is, due to the similarities of the time-harmonic bounds, essentially identical to the proof of \cite[Theorem~6.1]{NKL22}. We repeat the arguments given there and apply them to the present setting, to keep the paper self-contained. We structure the proof into three parts (a)--(c).

(a) {\it (Discretized time-harmonic boundary integral equation).\/} We start with the time-harmonic boundary integral equation \eqref{eq:bie-A-weak}, for $\Re s  > 0$.
We denote by $\bL_h(s):\mXG^4\to \X_h^4$ the solution operator $\widehat \bg\mapsto \widehat{\bphi_h}$ of the Galerkin approximation in $\X_h^4$, 
\begin{align}\label{bie-A-tau-h-s-g}
\left[
 \bxi_h ,
\bA(s) \widehat\bphi_h
\right]_\Gamma=
\left[
\bxi_h,
\bwg\right]_\Gamma
%[\bupsilon_h, \bwginc]_\Gamma
\quad \quad \forall \,\bxi_h \in \X_h^4,
\end{align}
which by the bound of $\bA(s)$ in Lemma~\ref{lemma:A-boundedness}, the coercivity estimate of Lemma~\ref{lem:A-coercivity} and the Lax--Milgram lemma yields, for $\Re s\ge \sigma >0$, the bound
\begin{equation}\label{L-h-bound}
\| \bL_h(s) \|_{ \X_h^4 \leftarrow \mXG^4} \le C_{\sigma} \dfrac{\abs{s}^2}{\Re s},
%\frac 1{c_\sigma} \, \frac{|s|^2}{\Re s} .
\end{equation}
where $C_{\sigma}$ depends on the surface $\Gamma$ and $\sigma$.
%\textcolor{blue}{where $m_{\epsmu}(s)$ was defined in \eqref{eq:max-term-in-bounds}. If stronger passivity holds, we have
%\begin{equation}\label{L-h-bound-sp}
%\| \bL_h(s) \|_{ \X_h^4 \leftarrow \mXG^4} \le c_\Gamma^{2}C\epsmu
%\,\frac{|s|^2}{\Re s}.
%\end{equation} }
The associated Ritz projection $\bR_h(s):\mXG^4 \to \X_h^4$ maps $(\widehat \bphi)\in\mXG^4$ to $\bwphi_h\in \X_h^4$, determined by
\begin{align*}%\label{bie-A-tau-h-s}
\left[
 \bxi_h
,
\bA(s)\bwphi_h
\right]_\Gamma=
\left[
\bxi_h , 
\bA(s)\bwphi\right]_\Gamma
%[\bupsilon_h,\g\inc]_\Gamma
\quad \quad \forall \,\bxi_h \in \X_h^4.
\end{align*}
Again by Lemmas~\ref{lemma:A-boundedness} and~\ref{lem:A-coercivity} and the Lax--Milgram lemma, this problem has a unique solution $(\bwvarphi_h,\bwpsi_h)\in \X_h^4$, and by C\'ea's lemma,
%\textcolor{blue}{
%\begin{equation*}
%\left\| \begin{pmatrix}\bwvarphi_h \\ \bwpsi_h\end{pmatrix} - \begin{pmatrix}\bwvarphi \\ \bwpsi\end{pmatrix} \right\|_{\mXG^4}
%%
%\le \Ctrace\ctrace^2 m\epsmu(s)\left(m\epsmu(s) + \frac12\right)
%\inf_{(\bupsilon_h,\bxi_h)\in \X_h^4}
%\left\| \begin{pmatrix}\bupsilon_h \\ \bxi_h\end{pmatrix} - \begin{pmatrix}\bwvarphi \\ \bwpsi\end{pmatrix} \right\|_{ \mXG^4},
%\end{equation*}}
%$$
%\left\| \begin{pmatrix}\bwvarphi_h \\ \bwpsi_h\end{pmatrix} - \begin{pmatrix}\bwvarphi \\ \bwpsi\end{pmatrix} \right\|_{\mXG^4}
%\le \frac{C_\sigma}{c_\sigma}\left(\frac{|s|^2}{\Re s} \right)^2
%\inf_{(\bupsilon_h,\bxi_h)\in \X_h^4}
%\left\| \begin{pmatrix}\bupsilon_h \\ \bxi_h\end{pmatrix} - \begin{pmatrix}\bwvarphi \\ \bwpsi\end{pmatrix} \right\|_{ \mXG^4},
%$$
where the right-hand side is further bounded by Lemma~\ref{lem:RT}.
With the stronger passivity, we arrive at the bound
\begin{equation*}
	\left\| \bwphi_h- \bwphi \right\|_{\mXG^4}
	\le \frac{C_\sigma}{c_\sigma} \left( \frac{|s|^2}{\Re s} \right)^2
	\inf_{\bxi_h\in \X_h^4}
	\left\|  \bxi_h - \bwphi\right\|_{ \mXG^4},
\end{equation*}
for all $\Re s \ge \sigma > 0$.

In combination with the approximation result of Lemma~\ref{lem:RT}, we can thus bound the associated error operator $\mathbfcal{E}_h(s) = \bR_h(s)-\bId$ in the operator norm from $\H^{k+1}_\times(\Gamma)^4$ to $\mXG^4$ with the bound, for $\Re s \ge \sigma > 0$,
\begin{equation}\label{Err-h-bound}
\| \mathbfcal{E}_h(s) \|_{\mVG\times\mXG \leftarrow \H^{k+1}_\times(\Gamma)^2} \le \widetilde C_\sigma \frac{|s|^4}{(\Re s)^2}\, h^{k+3/2} .
\end{equation}

(b) {\it (Error of the spatial semi-discretization).}
We continue with the spatial semi-discretization of the time-dependent boundary integral equation \eqref{bie-A-t}, which reads
\begin{align}\label{bie-A-h}
\left[
\bxi_h
,
\A(\pt)\bphi \right]
=\left[  \bxi_h
,\bg\right]_\Gamma \quad \quad \forall \,\bxi_h \in \X_h^4.
\end{align}
This formulation has the unique solution
$$
\bphi_h = \bL_h(\pt)\bg = \bR_h(\pt) \bphi ,
$$
where $\bphi=\A^{-1}(\pt)\bg$ is the solution of \eqref{bie-A-t}.
With the exterior and interior potential operators collected in the block operators $\mathbfcal{W}^{\pm}(s)$
%We abbreviate
%$$
%\W\epsmu(s) = \begin{pmatrix} -\sqrt{\frac{\mu}{\eps}}(s)\mathcal{S}\epsmu(s) & \mathcal{D}\epsmu(s) \\
%-\mathcal{D}\epsmu(s) & -\sqrt{\frac{\eps}{\mu}}(s)\mathcal{S}\epsmu(s)
%\end{pmatrix},
%$$
and the auxiliary operators $\mathbfcal P^{\pm}$ and $\mathbfcal N$ defined in \eqref{eq:auxiliary-operators} we set
\begin{equation}\label{U-h}
\U^{\pm}_h(s) = \W^{\pm}(s)\mathbfcal P^{\pm}\bL_h(s) \mathbfcal N \colon \mXG^2 \to \H(\curl,\Omega^{\pm})^2.
\end{equation}
With the established bounds from Lemma~\ref{lem:transmission} and \eqref{L-h-bound}, this operator family is bounded by
\begin{equation}\label{U-h-bound}
\| \U^{\pm}_h(s) \|_{\H(\curl,\Omega^{\pm})^2 \leftarrow \mXG^2} \le \bar C_\sigma\,  \frac{|s|^4}{(\Re s)^2}.
\end{equation}
The spatial semi-discretization of the scattering problem is then the forward convolution of $ \U^{\pm}_h(\pt)$ with the incident wave, which reads
$$
\begin{pmatrix} \E^{\pm}_h \\ \B^{\pm}_h
\end{pmatrix}
= \U^{\pm}_h(\pt) \bginc.
$$
In view of \eqref{U-t}, its error is
\begin{align*}
\begin{pmatrix} \E^{\pm}_h \\ \B^{\pm}_h
\end{pmatrix} -
\begin{pmatrix} \E^{\pm} \\  \B^{\pm}
\end{pmatrix}
&=
\U^{\pm}_h(\pt)\bginc - \U^{\pm}(\pt)\bginc 
=
 \W^{\pm}(\pt)\bphi^{\pm}_h- \W^{\pm}(\pt)\bphi^{\pm}
\\
&= \W^{\pm}(\pt)(\bR_h - \bId)\bphi^{\pm}= \W^{\pm}(\pt) \,\mathbfcal{E}_h(\pt) \bphi^{\pm}.
\end{align*}
Using the bound of Lemma~\ref{lem:transmission} for the potential operator $\W^{\pm}(s)$, the bound \eqref{Err-h-bound} for the error operator $\mathbfcal{E}_h(s)$, and the temporal Sobolev bound stated in \cite[Lemma~2.1]{L94} (with $\kappa=6$) for their composition, and finally the Sobolev embedding
$H^1(0,T;H)\subset C([0,T],H)$ for any Hilbert space $H$, we obtain for the error of the spatial semi-discretization
\begin{align}\label{error-h}
&\max_{0\le t \le T} \left\| \begin{pmatrix} \E^{\pm}_h(t) \\  \B^{\pm}_h(t)
\end{pmatrix} -
\begin{pmatrix} \E^{\pm}(t) \\  \B^{\pm}(t)
\end{pmatrix}
\right\|_{\H(\curl,\Omega^{\pm})^2}
\\
\nonumber
&\le
C  \left\| \begin{pmatrix} \E^{\pm}_h \\  \B^{\pm}_h
\end{pmatrix} -
\begin{pmatrix} \E^{\pm} \\  \B^{\pm}
\end{pmatrix}
\right\|_{H^1_0(0,T;\H(\curl,\Omega^{\pm})^2)}
\le C_T \, h^{k+3/2} \left\| \bphi^{\pm} \right\|_{H^7_0(0,T;\H^{k+1}_\times(\Gamma)^2)}.
\end{align}
Using the same argument with the pointwise bounds away from the boundary given by Lemma~\ref{lem:bound-away-boundary-deriv}, we further obtain
\begin{equation}\label{error-h-d}
\max_{0\le t \le T} \left\| \begin{pmatrix} \E^{\pm}_h(t) \\  \B^{\pm}_h(t)
\end{pmatrix} -
\begin{pmatrix} \E^{\pm}(t) \\  \B^{\pm}(t)
\end{pmatrix}
\right\|_{\bC^1(\overline\Omega^{\pm}_d)^2}
%\\
%\nonumber
\le C_T \, h^{k+3/2} \left\| \bphi^{\pm} \right\|_{\H^{\textcolor{blue}{9}}_0(0,T;\H^{k+1}_\times(\Gamma)^4)}.
\end{equation}

(c) {\it (Error of the full discretization).\/}  The total error is (omitting here the omnipresent superscript $n$)
\begin{equation}\label{total-error-split}
%\begin{pmatrix} \E^n_h \\  \B^n_h
%\end{pmatrix}
%-
%\begin{pmatrix} \E(t_n) \\  \B(t_n)
%\end{pmatrix}
%=
%\left(
\begin{pmatrix} \E^{\pm}_{\tau,h} \\  \H^{\pm}_{\tau,h}
\end{pmatrix}
-
\begin{pmatrix} \E^{\pm}_{\tau}\ \\  \E^{\pm}_{\tau}
\end{pmatrix}
%\right)
\quad\
+
\quad\
%\left(
\begin{pmatrix} \E^{\pm}_{\tau} \\  \E^{\pm}_{\tau}
\end{pmatrix}
-
\begin{pmatrix} \E^{\pm}\\  \B^{\pm}
\end{pmatrix}
%\right).
.
\end{equation}
The second difference is the error of the temporal semi-discretization, which is bounded by applying Lemma~\ref{rkcq} with the time-harmonic bounds on $\U^{\pm}(s)$ due to
\begin{equation}\label{eq:temporal-semidis}
	\begin{pmatrix} \E^{\pm}_\tau \\ \B^{\pm}_\tau
	\end{pmatrix}
	-
	\begin{pmatrix} \E^{\pm} \\  \B^{\pm}
\end{pmatrix}	
	= \left( \U^{\pm}(\pttau)-\U^{\pm}(\pt) \right) \bginc .
\end{equation}
With the time-harmonic bound \eqref{eq:U-Hcurl-bound}, we obtain an estimate of the order $O(\tau^{m-1/2})$ in the $\H(\curl,\Omega^{\pm})$-norm, whereas applying the time-harmonic \eqref{eq:U-pointwise-bound} yields an error estimate of the order $O(\tau^{2m-1})$ in the $\bC^1(\overline\Omega_d^{\pm})$-norm.

The first difference of the total error \eqref{total-error-split} is rewritten as 
\begin{align}\label{eq:final-split}
\begin{split}
\W^{\pm}(\pttau)\mathbfcal{E}_h(\pttau) \bphi^{\pm} &=
\left(\W^{\pm}(\pttau)\mathbfcal{E}_h(\pttau) \bphi^{\pm} -
\W^{\pm}(\pt)\mathbfcal{E}_h(\pt) \bphi^{\pm}\right)
\\
&+ \
\W^{\pm}(\pt)\mathbfcal{E}_h(\pt) \bphi^{\pm}.
\end{split}
\end{align}
The final error term is the spatial semi-discretization studied in part~(b),  which is therefore bounded by \eqref{error-h}. To bound the remaining difference, which is a convolution quadrature error, we employ  Lem\-ma~\ref{lem:RK-CQ}. This gives an $O(h^{k+3/2})$ error in the $\H(\curl,\Omega^{\pm})^2$ norm, using that by Lemma~\ref{lem:transmission} and \eqref{Err-h-bound} we have here $M_\sigma\le C_\sigma h^{k+3/2}$, $\kappa=6$, $\nu=3$ in \eqref{KLM-bound} with $\W(s)\mathbfcal{E}_h(s)$ in the role of $\K(s)$,
 and choosing $q=2$ and $r=10>2q-1+\kappa$. Note that here $\min(2q-1,q+1-\kappa+\nu)=q-2=0$. Altogether, this yields the stated $O(\tau^{m-1/2}+h^{k+3/2})$ error bound in the $\H(\curl,\Omega^{\pm})^2$ norm.
%\textcolor{blue}{To prove the full-order convergence of the error we use the same identity, but due to the exponential decay with respect to the real part of $s$, we can instead use the bound of Lemma~\ref{lem:RK-CQ} with $\kappa=6, \nu=4$ and $q=1$. The required smoothness of the data is therefore at least $r=8>2q-1+\kappa$. We then have again $\min(2q-1,q+1-\kappa+\nu) =0$, which implies the bound of the operator of the first term of \eqref{eq:final-split}.}

The full-order error bound away from the boundary can be shown without requiring this additional assumption on $r$. To show this bound, we rewrite the error as
\begin{equation*}
	%\begin{pmatrix} \E^n_h \\  \B^n_h
	%\end{pmatrix}
	%-
	%\begin{pmatrix} \E(t_n) \\  \B(t_n)
	%\end{pmatrix}
	%=
	%\left(
	\begin{pmatrix} \left(\E^{\pm}_{\tau,h}\right)^n \\  \left(\H^{\pm}_{\tau,h}\right)^n
	\end{pmatrix}
	-
	\begin{pmatrix} \E^{\pm}_{h}(t_n)\ \\  \E^{\pm}_{h}(t_n)
	\end{pmatrix}
	%\right)
	\quad\
	+
	\quad\
	%\left(
	\begin{pmatrix} \E^{\pm}_{h}(t_n)\ \\  \E^{\pm}_{h}(t_n)
\end{pmatrix}
	-
	\begin{pmatrix} \E^{\pm}(t_n)\\  \B^{\pm}(t_n)
	\end{pmatrix}
	%\right).
	.
\end{equation*}
%$$
%%\begin{pmatrix} \E^n_h \\  \B^n_h
%%\end{pmatrix}
%%-
%%\begin{pmatrix} \E(t_n) \\  \B(t_n)
%%\end{pmatrix}
%%=
%%\left(
%\begin{pmatrix} \E^n_h \\  \B^n_h
%\end{pmatrix}
%-
%\begin{pmatrix} \E_h(t_n) \\  \B_h(t_n)
%\end{pmatrix}
%%\right)
%\quad\
%+
%\quad\
%%\left(
%\begin{pmatrix} \E_h(t_n) \\  \B_h(t_n)
%\end{pmatrix}
%-
%\begin{pmatrix} \E(t_n) \\  \B(t_n)
%\end{pmatrix}
%%\right).
%.
%$$
The second difference is the error of the spatial semi-discretization studied in part~(b). The first difference is a convolution quadrature error for the transfer operator $\U^{\pm}_h(s)$ of \eqref{U-h}:
$$
\begin{pmatrix} \left(\E^{\pm}_{\tau,h}\right)^n \\ \left( \B^{\pm}_{\tau,h}\right)^n
\end{pmatrix}
-
\begin{pmatrix} \E^{\pm}_h(t_n) \\  \B^{\pm}_h(t_n)
\end{pmatrix}
=
\Bigl[\bigl(\U^{\pm}_h(\pttau)\bginc \bigr)^{n-1}\Bigr]_m - \U^{\pm}_h(\pt)\bginc(t_n).
$$
Using this argument to bound the error in the $\H(\curl,\Omega^{\pm})$ norm by Lemma~\ref{lem:RK-CQ} would reduce the predicted error rate to $O(\tau^{m-1})$, hence the different argument structure before.
%From Lemma~\ref{lem:RK-CQ} with $U_h(s)$ bounded by \eqref{U-h-bound} in the role of $K(s)$,
%we then obtain the stated error bound in the $\H(\curl,\Omega)$ norm.

The exponential decay in the bound \eqref{eq:U-pointwise-bound} exceeds any polynomial decay, which gives with \eqref{L-h-bound} a constant $C_{\sigma,d}$, depending only on $\sigma$ and $d$, such that
\begin{align}\label{eq:bound-U-h-dispersive}
	\| \U^{\pm}_h(s) \|_{(\bC^1(\overline\Omega^{\pm}_d)^3)^2\leftarrow \mXG^2}
	\le C_{\sigma,d}\frac{\abs{s}^{\frac{9}{2}}}{\left(\Re s\right)^{\frac{3}{2}+m+1}}\,
	\norm{\bwginc}_{\mXG^2},
\end{align}
%\begin{align*}
%	\| \U^{\pm}_h(s) \|_{(\bC^1(\overline\Omega^{\pm}_d)^3)^2\leftarrow \mXG^2}
%	\le C_{\sigma,d}\frac{\abs{s}^{4+1/2+\frac{m+1}{2}}}{\left(\Re s\right)^{1+1/2+3/2(m+1)}}\,
%	\norm{\bwginc}_{\mXG^2},
%\end{align*}
for $\Re s\ge \sigma>0$, by using $e^{-x}\le C x^{-m-1}$ for $x\ge \sigma$.
We then obtain the stated full convergence rates in the $\H(\curl,\Omega^{\pm}_d)$ norm and the $\bC^1(\overline\Omega^{\pm}_d)$ norm by Lemma~\ref{lem:RK-CQ}, with $r>2m+7/2$ and $\kappa,\nu$ chosen accordingly to the bound above. 

\end{proof}

\section{Numerical experiments}
\label{sec:numerics}
We complement the theory of the previous sections by the following experiments. The boundary element approximations of the  boundary and potential operators of the Maxwell problem were realized by the library Bempp, which is described in \cite{Bempp}. The codes used to generate the simulation data and the figures are available via github.\footnote{\href{https://github.com/joergnick/cqExperiments}{https://github.com/joergnick/cqExperiments}, last accessed on 25/10/2023.}

All experiments have been conducted with the following setting. One or several scatterers are illuminated by an incoming plane wave of the form
\begin{align}\label{eq:inc}
\Einc(\bx,t) =
\boldsymbol{p} e^{-c\norm{\boldsymbol{d}\cdot \bx +t-t_0}^2},
\end{align}
The polarization vector is set to $\boldsymbol{p} = \frac{1}{\sqrt{2}}(-1,0,-1)^T$, the direction to $\boldsymbol{d} =\frac{1}{\sqrt{2}}((-1,0,1)^T$ and the temporal shift to $t_0=4$.
 We observe the interaction of the wave with different scatterers until the final time $T=8$.
The physical constants in the exterior domain $\Omega^+$ are set to one, i.e. $\varepsilon^+=\mu^+=1$. Inside the obstacle, we enforce a fractional material law, which reads
\begin{align}\label{eq:fract-const}
	\varepsilon^-(s) = \frac{1}{2}+\frac{1}{1+s^{1/2}}, \quad 
	\mu^-(s) = \frac{1}{2}.
\end{align}
The corresponding time-varying material law includes fractional time derivatives and is therefore nonlocal in time. Moreover,  since $\varepsilon^-(s)$ is not a rational function, techniques based on memory variables are not available.

\subsection{Scattering from a sphere: Convergence plots}
To investigate empirical convergence rates, we consider the following simple setting. The exterior of a unit sphere centered at the origin is initially at rest and excited by the plane wave \eqref{eq:inc} with $c=10$. A sequence of grids, with the mesh widths $h_j=2^{-j/2}$ for $j=0,...,7$ is used with $0$-th order Raviart--Thomas elements as the space discretization. As the time discretization, we employ the convolution quadrature method based on the $2$-stage Radau IIA method, for $N_j=8\cdot2^j$ for $j=0,...,7$. The numerical approximations are compared with a reference solution obtained by the same discretization, that has been computed with $h=2^{-4}$, which corresponds to a boundary element space of $12534$ degrees of freedom, and $N=2048$ time steps.

\begin{figure}[htbp]
	%\label{Fig2}
	\centering
	\includegraphics[trim = 0mm 0mm 0mm 0mm, clip,width=1.0\textwidth,height=0.55\textwidth]{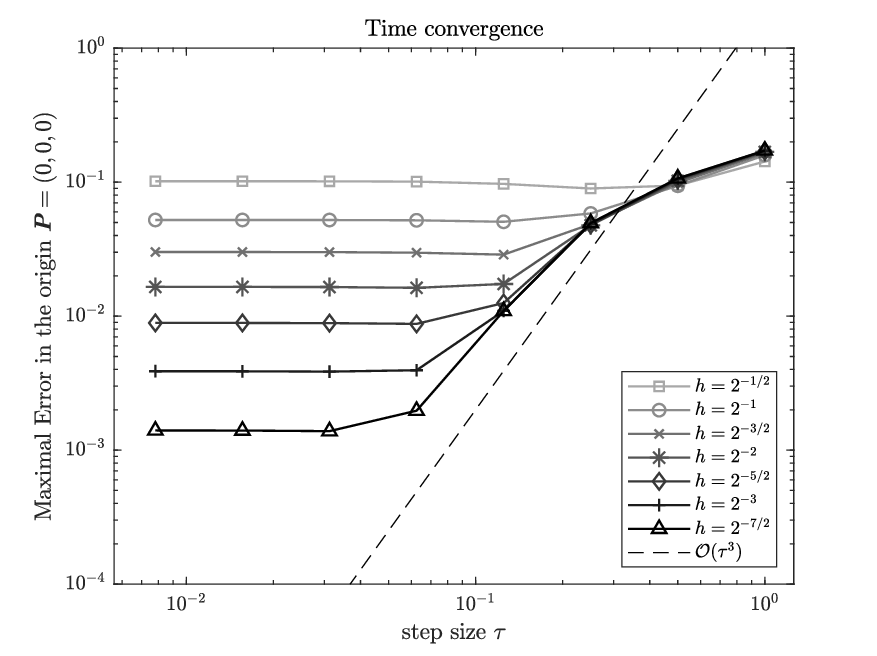}
	\caption{Time convergence plot of the fully discrete system for a spherical scatterer, for $0$th order Raviart--Thomas boundary elements and the $2$-stage Radau IIA based Runge--Kutta convolution quadrature method.}
	\label{fig:time_conv}
\end{figure}

\begin{figure}[htbp]
	%\label{Fig2}
	\centering
	\includegraphics[trim = 0mm 0mm 0mm 0mm, clip,width=1.0\textwidth,height=0.55\textwidth]{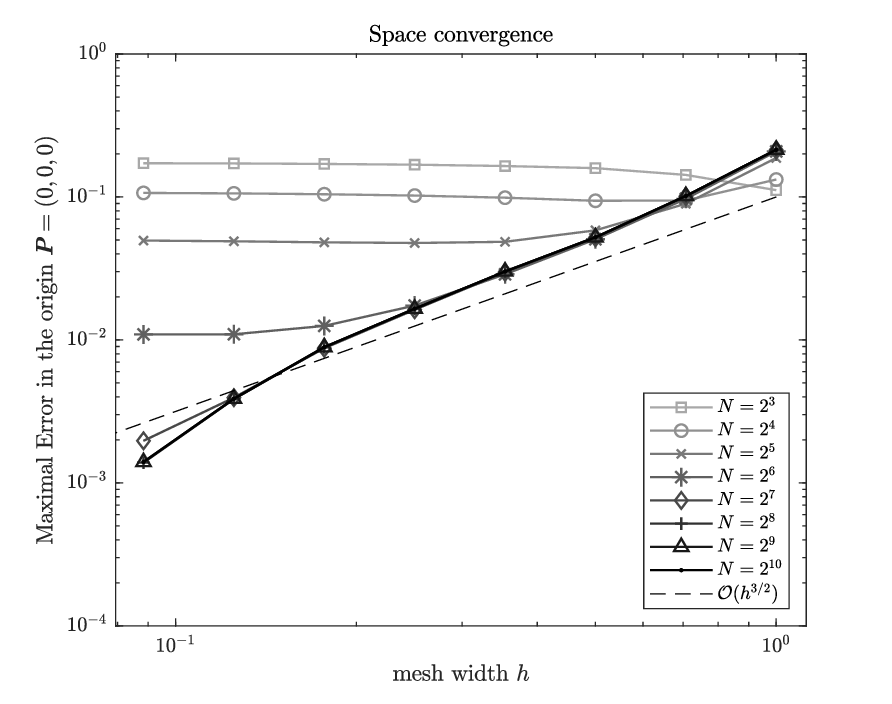}
	\caption{Space convergence plot of the fully discrete system for a spherical scatterer, for $0$th order Raviart--Thomas boundary elements and the $2$-stage Radau IIA based Runge--Kutta convolution quadrature method. 
	}
	\label{fig:space_conv}
\end{figure}

\subsection{Scattering from two cubes: Visualization of the numerical solution}
In the second experiment, we choose the union of two unit cubes, separated by a gap of length $l=0.5$, as the interior domain $\Omega^-$. The plane wave \eqref{eq:inc} illuminates the scatterers,  and $c=100$.
Figure~\ref{fig:frames} depicts the approximation of the total wave, evaluated in the $y=0.5$ plane, which cuts through the middle of the cubes, at several time points. The scheme has been used with a $0-th$ order Raviart--Thomas boundary element discretization with 
$11088$ degrees of freedom and the convolution quadrature time discretization based on the $2$--stage Radau IIA method with $N=2096$ time steps.

\begin{figure}
	\centering
\vspace*{-0.4cm}
	\includegraphics[trim = 0mm 0mm 0mm 5mm,clip]{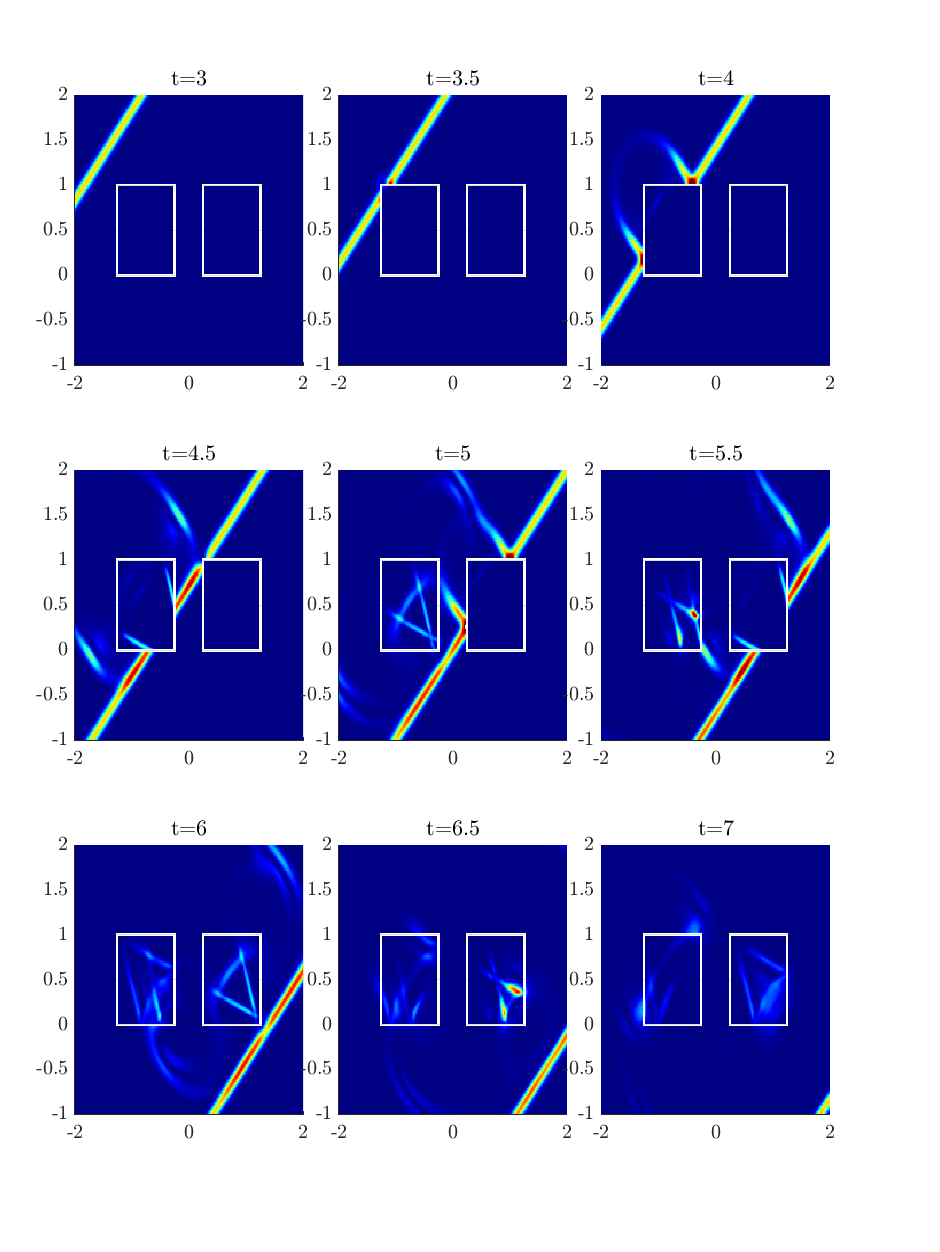}
	\vspace*{-2cm}
	\caption{Scattering from two cubes, where the time-fractional material law determined by \eqref{eq:fract-const} is enforced.}
	\label{fig:frames}
\end{figure}

\section*{Acknowledgement} This work was supported by the Deutsche Forschungsgemeinschaft (DFG, German Research Foundation) -- Project-ID 258734477 -- SFB 1173. We thank Marlis Hochbruck for turning our attention to Maxwell's equations with dispersive material laws.

\bibliographystyle{IMANUM-BIB}
	\bibliography{Lit}
\end{document}